\documentclass[a4paper,12pt]{article}
\usepackage{ucs}
\usepackage{amsfonts, amssymb, amsmath, amsthm, amscd}
\usepackage{graphicx}
\usepackage{cite}
\ifx\undefined \pdfgentounicode \else
\input{glyphtounicode} \pdfgentounicode=1
\fi

\author{A. A. Vasil'eva}
\title{Estimates for the Kolmogorov widths of an intersection of two balls in a mixed norm}
\date{}
\begin{document}

\maketitle

\newenvironment{Biblio}{%
                  \renewcommand{\refname}{\footnotesize REFERENCES}%
                  \begin{thebibliography}{99}%
                  \itemsep -0.2em}%
                  {\end{thebibliography}}

\def\inff{\mathop{\smash\inf\vphantom\sup}}
\renewcommand{\le}{\leqslant}
\renewcommand{\ge}{\geqslant}
\newcommand{\sgn}{\mathrm {sgn}\,}
\newcommand{\inter}{\mathrm {int}\,}
\newcommand{\dist}{\mathrm {dist}}
\newcommand{\supp}{\mathrm {supp}\,}
\newcommand{\R}{\mathbb{R}}
\newcommand{\Z}{\mathbb{Z}}
\newcommand{\N}{\mathbb{N}}
\newcommand{\Q}{\mathbb{Q}}
\theoremstyle{plain}
\newtheorem{Trm}{Theorem}
\newtheorem{trma}{Theorem}

\newtheorem{Def}{Definition}
\newtheorem{Cor}{Corollary}
\newtheorem{Lem}{Lemma}
\newtheorem{Rem}{Remark}
\newtheorem{Sta}{Proposition}

\newtheorem{Not}{Notation}
\newtheorem{Exa}{Example}
\renewcommand{\proofname}{\bf Proof}
\renewcommand{\thetrma}{\Alph{trma}}

\begin{abstract}
Order estimates for the Kolmogorov widths of an intersection of two finite‐dimensional balls in a mixed norm under some conditions on the parameters are obtained.
\end{abstract}

\section{Introduction}

In this paper, a problem of order estimates for the Kolmogorov widths of an intersection of two finite‐dimensional balls in a mixed norm is studied.

First we give necessary definitions and notation.

Let $m, \, k\in \N$, $1\le p<\infty$, $1\le \theta<\infty$. By $l_{p,\theta}^{m,k}$ we denote the space $\R^{mk}$ with the norm
$$
\|(x_{i,j})_{1\le i\le m, \, 1\le j\le k}\|_{l_{p,\theta}^{m,k}} = \left(\sum \limits _{j=1}^k\left(\sum \limits _{i=1}^m |x_{i,j}|^p\right)^{\theta/p}\right)^{1/\theta}.
$$
For $p=\infty$ or $\theta=\infty$, the definition  is modified naturally.

By $B_{p,\theta}^{m,k}$ we denote the unit ball of the space $l_{p,\theta}^{m,k}$. If $k=1$, we write $l_p^m:=l_{p,\theta}^{m,1}$ and $B_p^m:=B_{p,\theta}^{m,1}$.

Let $X$ be a normed space, and let $M\subset X$, $n\in
\Z_+$. The Kolmogorov $n$-widts of the set $M$ in the space $X$ is defined as follows:
$$
d_n(M, \, X) = \inf _{L\in {\cal L}_n(X)} \sup _{x\in M} \inf
_{y\in L} \|x-y\|;
$$
here ${\cal L}_n(X)$ is the family of all subspaces in $X$
with dimension at most $n$. For details, see, e.g., \cite{itogi_nt, kniga_pinkusa, nvtp}.

Exact values of the widths $d_n(B_p^m, \, l_q^m)$ were obtained in \cite{pietsch1, stesin} (for $p\ge q$) and in  \cite{k_p_s, stech_poper} (for $p=1$, $q=2$). For $p\le q<\infty$, order estimates were obtained in \cite{gluskin1, bib_gluskin}. The problem of estimating the widths $d_n(B_p^m, \, l_\infty^m)$ was studied in \cite{kashin_oct, bib_kashin, garn_glus}; for $p\ge 2$, the order estimates were obtained; for $1\le p<2$, the values are known up to a factor, which is a degree of $\log\left(\frac{em}{n}\right)$.

Approximative properties of the balls $B^{m,k}_{p,\theta}$ in $l^{m,k}_{q,\sigma}$ are interesting in relation to  
Besov classes with dominating mixed smoothness \cite{galeev4, dir_ull, vyb_06} and weighted Besov classes \cite{vas_besov}.
In \cite{galeev2, galeev5, izaak1, izaak2, mal_rjut, vas_besov, dir_ull}, the problem of estimating the Kolmogorov widths $d_n(B^{m,k}_{p,\theta}, \, l^{m,k}_{q,\sigma})$ for $n \le \frac{mk}{2}$ was studied (more precisely, in \cite{dir_ull} the Gelfand widths were considered; for $p$, $\theta$, $q$, $\sigma \ge 1$, the problem can be formulated in terms of the Kolmogorov widths \cite{ioffe_tikh}). The order estimates were obtained for the following parameters:
\begin{enumerate}
\item E. M. Galeev \cite{galeev2}: $p=1$, $\theta=\infty$, $q=2$, $1<\sigma <\infty$;
\item E. M. Galeev \cite{galeev5}: $p=1$ or $p=\infty$; $\theta=\infty$; here one of the following conditions holds: a) $q=2$, $1<\sigma \le \infty$ or b) $1<q\le \min \{2, \, \sigma\}$;
\item A. D. Izaak \cite{izaak2}: $p=\theta$, $q=2$, $\sigma=1$, where $p=1$ or $2\le p\le \infty$;
\item in \cite{vas_besov} the case $2\le q<\infty$, $2\le \sigma <\infty$, $1\le p\le q$, $1\le \theta \le \sigma$, $n\le a(q, \, \sigma)mk$ was considered (here $a(q, \, \sigma)$ is a positive number);
\item Yu. V. Malykhin, K. S. Rjutin \cite{mal_rjut}: $p=1$, $\theta=\infty$, $q=2$, $\sigma =1$ (earlier in \cite{izaak1} the estimates were obtained up to a logarithmic factor), as well as $p\le q\le 2$, $\theta \ge \sigma$;
\item S. Dirksen, T. Ullrich \cite{dir_ull}: a) $p=q=2$, $\theta\ge 2$, $\sigma=\infty$; b) $p=\theta=\sigma\ge 2$, $q=\infty$.
\end{enumerate}

In addition, E. M. Galeev \cite{galeev6} obtained the lower estimate of the Kolmogorov widths for $1\le p\le \infty$, $\theta=\infty$, $2\le q<\infty$, $\sigma=q$, $n\le c(q)mk$ (here $c(q)$ is a positive number).

The problem of estimating the Kolmogorov widths of an intersection of a family of Sobolev or Besov classes \cite{galeev1, galeev2, galeev4} can be reduced by the discretization method to estimating the widths of $d_n(\cap _{\alpha \in A} \nu_\alpha B^m_{p_\alpha}, \, l_q^m)$. E. M. Galeev \cite{galeev1} obtained the order estimates of these values for $n= \frac{m}{2}$; in \cite{vas_ball_inters} this result was generalized to $n \le \frac{m}{2}$.

Naturally arises the question of estimating the Kolmogorov widths of an intersection of finite‐dimension balls in a mixed norm. The result can be employed in estimating the widths of an intersection of weighted Besov classes or Besov classes with dominating mixed smoothness. Here we consider the case of two balls $\nu_i B^{m,k}_{p_i,\theta_i}$, $i=1, \, 2$, where $2\le q<\infty$, $2\le \sigma <\infty$, $1\le p_i\le q$, $1\le \theta_i\le \sigma$, $i=1, \, 2$. It turns out that for these parameters the problem can be reduced to estimating the widths of one ball in a mixed norm; the order estimates for such widths are already obtained in \cite{vas_besov} (see Theorem \ref{theorem_a} below).

Given sets $X$, $Y$ and functions $f_1$, $f_2:\ X\times Y\rightarrow \mathbb{R}_+$, we write
$f_1(x, \, y)\underset{y}{\lesssim} f_2(x, \, y)$ (or
$f_2(x, \, y)\underset{y}{\gtrsim} f_1(x, \, y)$) if for each
$y\in Y$ there exists $c(y)>0$ such that $f_1(x, \, y)\le
c(y)f_2(x, \, y)$ for all $x\in X$; $f_1(x, \,
y)\underset{y}{\asymp} f_2(x, \, y)$ if $f_1(x, \, y)
\underset{y}{\lesssim} f_2(x, \, y)$ and $f_2(x, \,
y)\underset{y}{\lesssim} f_1(x, \, y)$.

Let $q>2$, $1\le p\le q$. We set $\lambda_{p,q} = \min \left\{\frac{1/p‐1/q}{1/2‐1/q}, \, 1\right\}$. For $q=2$, $1\le p\le 2$, we set $\lambda_{p,2}=1$.

\begin{trma}
\label{theorem_a}
(see \cite{vas_besov}). Let $m$, $k\in \N$, $n\in \Z_+$, $n\le \frac{mk}{2}$, $2\le q<\infty$, $2\le \sigma <\infty$, $1\le p\le q$, $1\le \theta \le \sigma$. Then
\begin{itemize}
\item if $\max\{p, \, \theta\}\le 2$, then
\begin{align}
\label{dn_1} d_n(B^{m,k}_{p,\theta}, \, l_{q,\sigma}^{m,k}) \underset{q,\sigma}{\asymp} \min\{1, \, n^{‐\frac 12}m^{\frac 1q} k^{\frac{1}{\sigma}}\};
\end{align}
\item if $\max\{p, \, \theta\}\ge 2$, $\lambda_{p,q}\le \lambda _{\theta,\sigma}$, then
\begin{align}
\label{dn_2}
d_n(B^{m,k}_{p,\theta}, \, l_{q,\sigma}^{m,k}) \underset{q,\sigma}{\asymp} \begin{cases} 1, & n\le m^{\frac 2q}k^{\frac{2}{\sigma}}, \\ \left(n^{‐\frac 12}m^{\frac 1q}k^{\frac{1}{\sigma}}\right)^{\lambda_{p,q}}, & m^{\frac 2q}k^{\frac{2}{\sigma}}\le n \le mk^{\frac{2}{\sigma}},\\ m^{\frac 1q‐\frac 1p}(n^{‐\frac 12}m^{\frac 12}k^{\frac{1}{\sigma}})^{\lambda_{\theta,\sigma}},  & mk^{\frac{2}{\sigma}} \le n \le \frac{mk}{2};\end{cases} 
\end{align}
\item if $\max\{p, \, \theta\}\ge 2$, $\lambda_{p,q}\ge \lambda _{\theta,\sigma}$, then
\begin{align}
\label{dn_3}
d_n(B^{m,k}_{p,\theta}, \, l_{q,\sigma}^{m,k}) \underset{q,\sigma}{\asymp} \begin{cases} 1, & n\le m^{\frac 2q}k^{\frac{2}{\sigma}}, \\ \left(n^{‐\frac 12}m^{\frac 1q}k^{\frac{1}{\sigma}}\right)^{\lambda_{\theta,\sigma}}, & m^{\frac 2q}k^{\frac{2}{\sigma}}\le n \le km^{\frac{2}{q}},\\ k^{\frac{1}{\sigma}‐\frac{1}{\theta}}(n^{‐\frac 12}k^{\frac 12}m^{\frac{1}{q}}) ^{\lambda_{p,q}},  & km^{\frac{2}{q}} \le n \le \frac{mk}{2}.\end{cases} 
\end{align}
\end{itemize}
\end{trma}

In \cite{vas_besov} this theorem was proved for $n \le a(q, \, \sigma)mk$; in addition, in the statement, the constants in order equalities depend on $p$, $\theta$, $q$, $\sigma$, but the proof shows that they are independnet of $p$ and $\theta$. The upper estimate holds for all $n\le mk$. For $a(q, \, \sigma)mk \le n \le \frac{mk}{2}$ the lower estimate will be proved in \S 2 (see Corollary \ref{cor1}).

Notice that if $2\le p\le q$, $2\le \theta\le \sigma$, $\lambda_{p,q} = \lambda_{\theta,\sigma}$, then
\begin{align}
\label{1234} \begin{array}{c} \left(n^{‐\frac 12}m^{\frac 1q}k^{\frac{1}{\sigma}}\right)^{\lambda_{p,q}} = m^{\frac 1q‐\frac 1p}(n^{‐\frac 12}m^{\frac 12}k^{\frac{1}{\sigma}})^{\lambda_{\theta,\sigma}} \\ = \left(n^{‐\frac 12}m^{\frac 1q}k^{\frac{1}{\sigma}}\right)^{\lambda_{\theta,\sigma}} = k^{\frac{1}{\sigma}‐\frac{1}{\theta}}(n^{‐\frac 12}k^{\frac 12}m^{\frac{1}{q}}) ^{\lambda_{p,q}}. \end{array}
\end{align}

Now we formulate the main result of the article.
\begin{Trm}
\label{main}
Let $m$, $k\in \N$, $n\in \Z_+$, $n\le \frac{mk}{2}$, $2\le q<\infty$, $2\le \sigma <\infty$, $1\le p_i\le q$, $1\le \theta_i\le \sigma$, $\nu_i>0$, $i=1, \, 2$. We define the values $\Phi_j(m, \, k, \, n) = \Phi_j(m, \, k, \, n; \, p_1, \, p_2, \, \theta_1, \, \theta_2, \, q, \, \sigma, \, \nu_1, \, \nu_2)$ ($j=1, \, \dots, \, 5$) as follows:
\begin{enumerate}
\item $\Phi_j(m, \, k, \, n) = \nu_j d_n(B^{m,k}_{p_j,\theta_j}, \, l^{m,k}_{q,\sigma})$ for $j=1, \, 2$;
\item if there exists $\tilde \lambda \in [0, \, 1]$ such that $\frac 12 = \frac{1‐\tilde \lambda}{p_1} + \frac{\tilde \lambda}{p_2}$, we define the number $\tilde \theta$ by $\frac{1}{\tilde \theta} = \frac{1‐\tilde \lambda}{\theta_1} + \frac{\tilde \lambda}{\theta_2}$ and set $\Phi_3(m, \, k, \, n) = \nu_1^{1‐\tilde \lambda}\nu_2^{\tilde \lambda} d_n(B^{m,k}_{2,\tilde\theta}, \, l^{m,k}_{q,\sigma})$; otherwise, we set $\Phi_3(m, \, k, \, n)=+\infty$;
\item if there exists $\tilde \mu \in [0, \, 1]$ such that $\frac 12 = \frac{1‐\tilde \mu}{\theta_1} + \frac{\tilde \mu}{\theta_2}$, we define the number $\tilde p$ by $\frac{1}{\tilde p} = \frac{1‐\tilde \mu}{p_1} + \frac{\tilde \mu}{p_2}$ and set $\Phi_4(m, \, k, \, n) = \nu_1^{1‐\tilde \mu}\nu_2^{\tilde \mu} d_n(B^{m,k}_{\tilde p,2}, \, l^{m,k}_{q,\sigma})$; otherwise, we set $\Phi_4(m, \, k, \, n)=+\infty$;
\item if there exist $\lambda \in [0, \, 1]$, $p\in [2, \, q]$, $\theta\in [2, \, \sigma]$ such that $\frac 1p = \frac{1‐\lambda}{p_1} + \frac{\lambda}{p_2}$, $\frac{1}{\theta} = \frac{1‐\lambda}{\theta_1} + \frac{\lambda}{\theta_2}$ and $\lambda_{p,q}=\lambda_{\theta,\sigma}$, we set $\Phi_5(m, \, k, \, n) = \nu_1^{1‐\lambda}\nu_2^{\lambda} d_n(B^{m,k}_{p,\theta}, \, l^{m,k}_{q,\sigma})$; otherwise, we set $\Phi_5(m, \, k, \, n)=+\infty$.
\end{enumerate}
Then
$$
d_n(\nu_1B^{m,k}_{p_1,\theta_1} \cap \nu_2 B^{m,k}_{p_2,\theta_2}, \, l^{m,k}_{q,\sigma}) \underset{q,\sigma}{\asymp} \min _{1\le j\le 5} \Phi_j(m,\, k, \, n).
$$
\end{Trm}

The result is announced in \cite{vas_mz}.

\section{Auxiliary results}

Let $k$, $m$, $r$, $l\in \N$, $1\le r\le m$, $1\le l\le k$. We set $$G=\{(\tau_1, \, \tau_2, \, \varepsilon_1, \, \varepsilon_2):\; \tau_1\in S_m, \, \tau_2\in S_k, \, \varepsilon_1\in \{1, \, ‐1\}^m, \, \varepsilon_2\in \{1, \, ‐1\}^k\},$$
where $S_m$ and $S_k$ are groups of permutations of $m$ and $k$ elements, respectively. For $x=(x_{i,j})_{1\le i\le m, \, 1\le j\le k}\in \R^{mk}$, $\gamma = (\tau_1, \, \tau_2, \, \varepsilon_1, \, \varepsilon_2)\in G$, $\varepsilon_1=(\varepsilon_{1,i})_{1\le i\le m}$, $\varepsilon_2=(\varepsilon_{2,j})_{1\le j\le k}$, we set 
\begin{align}
\label{gamma_x_def}
\gamma(x) = (\varepsilon_{1,i}\varepsilon_{2,j}x_{\tau_1(i)\tau_2(j)})_{1\le i\le m, \, 1\le j\le k}.
\end{align}

We write $e=(e_{i,j}^{m,k,r,l})_{1\le i\le m, \, 1\le j\le k}$, where
\begin{align}
\label{eij_kl}
e_{i,j}^{m,k,r,l} = \left\{ \begin{array}{l} 1 \quad \text{if } 1\le i\le r, \; 1\le j\le l, \\ 0 \quad \text{otherwise},\end{array} \right.
\end{align}
\begin{align}
\label{vrl_km}
V_{r,l}^{m,k} = {\rm conv}\{\gamma(e):\; \gamma\in G\}.
\end{align}

In \cite[formula (34)]{vas_besov} the following assertion was obtained: if $2\le q<\infty$, $2\le \sigma<\infty$, $n\in \Z_+$, $n\le a(q, \, \sigma) m^{\frac 2q}k^{\frac{2}{\sigma}}r^{1‐\frac 2q} l^{1‐\frac{2}{\sigma}}$, then
\begin{align}
\label{dn_vmk}
d_n(V^{m,k}_{r,l}, \, l^{m,k}_{q,\sigma}) \ge b(q, \, \sigma) r^{\frac 1q}l^{\frac{1}{\sigma}};
\end{align}
here $a(q, \, \sigma)>0$, $b(q, \, \sigma)>0$, $a(\cdot, \, \cdot)$ is a function nonincreasing in each argument, $b(\cdot, \, \cdot)$ is a continuous function. Here we obtain the estimate for all $n\le \frac{mk}{2}$. We apply the method from the paper of Gluskin \cite{gluskin1}.

\begin{Sta}
\label{v_dn}
Let $2\le q<\infty$, $2\le \sigma <\infty$, $n\in \Z_+$, $n\le \frac{mk}{2}$. Then
\begin{align}
\label{dn_vmk1} d_n(V^{m,k}_{r,l}, \, l^{m,k}_{q,\sigma}) \underset{q,\sigma}{\gtrsim} \begin{cases} r^{\frac 1q}l^{\frac{1}{\sigma}} & \text{if}\; n\le  m^{\frac 2q}k^{\frac{2}{\sigma}}r^{1‐\frac 2q} l^{1‐\frac{2}{\sigma}}, \\ n^{‐\frac 12}m^{\frac 1q}k^{\frac{1}{\sigma}} r^{\frac 12} l^{\frac 12} & \text{if}\; n\ge  m^{\frac 2q}k^{\frac{2}{\sigma}}r^{1‐\frac 2q} l^{1‐\frac{2}{\sigma}}.\end{cases}
\end{align}
\end{Sta}
\begin{proof}
For $n\le a(q, \, \sigma)m^{\frac 2q}k^{\frac{2}{\sigma}}r^{1‐\frac 2q} l^{1‐\frac{2}{\sigma}}$, the estimate follows from (\ref{dn_vmk}).

Let $a(q, \, \sigma)m^{\frac 2q}k^{\frac{2}{\sigma}}r^{1‐\frac 2q} l^{1‐\frac{2}{\sigma}} \le n\le a(q, \, \sigma)mk$. There exist numbers $\tilde q\in [2, \, q]$ and $\tilde \sigma \in [2, \, \sigma]$ such that 
\begin{align}
\label{n_a_q_sigma}
n = a(q, \, \sigma)m^{\frac{2}{\tilde q}}k^{\frac{2}{\tilde\sigma}}r^{1‐\frac{2}{\tilde q}} l^{1‐\frac{2}{\tilde\sigma}}. 
\end{align}
Since the function $a(\cdot, \, \cdot)$ is non‐increasing in each argument, we have $n\le a(\tilde q, \, \tilde \sigma)m^{\frac{2}{\tilde q}}k^{\frac{2}{\tilde\sigma}}r^{1‐\frac{2}{\tilde q}} l^{1‐\frac{2}{\tilde\sigma}}$. Hence, by (\ref{dn_vmk}),
$$
d_n(V^{m,k}_{r,l}, \, l^{m,k}_{\tilde q,\tilde \sigma}) \ge b(\tilde q, \, \tilde \sigma) r^{\frac{1}{\tilde q}}l^{\frac{1}{\tilde\sigma}} \underset{q,\sigma}{\gtrsim} r^{\frac{1}{\tilde q}}l^{\frac{1}{\tilde\sigma}}
$$
(here we used the fact that $b$ is continuous). This implies
$$
d_n(V^{m,k}_{r,l}, \, l^{m,k}_{q,\sigma}) \ge m^{\frac 1q‐\frac{1}{\tilde q}}k^{\frac{1}{\sigma} ‐\frac{1}{\tilde \sigma}}d_n(V^{m,k}_{r,l}, \, l^{m,k}_{\tilde q,\tilde \sigma}) \underset{q,\sigma}{\gtrsim}
$$
$$
\gtrsim m^{\frac 1q‐\frac{1}{\tilde q}}k^{\frac{1}{\sigma} ‐\frac{1}{\tilde \sigma}}r^{\frac{1}{\tilde q}}l^{\frac{1}{\tilde\sigma}} \stackrel{(\ref{n_a_q_sigma})}{\underset{q,\sigma}{\asymp}} m^{\frac 1q}k^{\frac{1}{\sigma}}n^{‐\frac 12}r^{\frac 12} l^{\frac{1}{2}}.
$$

It remains to consider $a(q, \, \sigma)mk \le n \le \frac{mk}{2}$. First we show that 
\begin{align}
\label{dn_vmk_22}
d_n(V^{m,k}_{r,l}, \, l^{m,k}_{2,2}) \gtrsim r^{\frac 12}l^{\frac 12} \quad \text{for}\; n\le \frac{mk}{2}.
\end{align}
To this end, we argue as in \cite[pp. 14--17]{vas_besov} for this particular case. Let $Y\subset l_{2,2}^{m,k}$ be a subspace of dimension at most $n$, and let $y^\gamma =(y^\gamma_{i,j})_{1\le i\le m, \, 1\le j\le k}\in Y$ be a nearest point from $Y$ to $\gamma(e)$ (see (\ref{gamma_x_def}), (\ref{eij_kl})). Then
$$
\|\gamma(e)‐y^\gamma\|^2_{l_{2,2}^{m,k}} = \sum \limits _{j=1}^k \sum \limits _{i=1}^m |\gamma(e)_{i,j} ‐ y^\gamma_{i,j}|^2 = 
$$
$$
= \sum \limits _{j=1}^k \sum \limits _{i=1}^m |\gamma(e)_{i,j}|^2 + \sum \limits _{j=1}^k \sum \limits _{i=1}^m |y^\gamma_{i,j}|^2 ‐ 2 \sum \limits _{j=1}^k \sum \limits _{i=1}^m \gamma(e)_{i,j}y^\gamma_{i,j}.
$$
Averaging over $\gamma\in G$ and taking into account formulas (\ref{gamma_x_def}), (\ref{eij_kl}), we get
$$
\sup _{\gamma \in G}\|\gamma(e)‐y^\gamma\|^2_{l_{2,2}^{m,k}}\ge rl + \frac{1}{|G|} \sum \limits _{\gamma\in G}\sum \limits _{j=1}^k \sum \limits _{i=1}^m |y^\gamma_{i,j}|^2‐ 
$$
$$
‐2\frac{1}{|G|} \sum \limits _{\gamma\in G}\sum \limits _{j=1}^k \sum \limits _{i=1}^m \gamma(e)_{i,j}y^\gamma_{i,j}=:S.
$$
In \cite[pp. 16]{vas_besov} it was proved that
$$
\left|\frac{1}{|G|} \sum \limits _{\gamma\in G}\sum \limits _{j=1}^k \sum \limits _{i=1}^m \gamma(e)_{i,j}y^\gamma_{i,j}\right| \le \left(\frac{nrl}{mk}\right)^{1/2} \xi,
$$
where $\xi=\left(\frac{1}{|G|} \sum \limits _{\gamma\in G}\sum \limits _{j=1}^k \sum \limits _{i=1}^m |y^\gamma_{i,j}|^2\right)^{1/2}$.
Hence
$$
S\ge rl ‐2\left(\frac{nrl}{mk}\right)^{1/2}\xi + \xi^2 \ge rl \left(1‐\frac{n}{mk}\right).
$$
For $n\le \frac{mk}{2}$ we get $S\ge \frac{rl}{2}$; this together with (\ref{vrl_km}) implies (\ref{dn_vmk_22}).

Let now $q\in [2, \, \infty)$, $\sigma \in [2, \infty)$. Then
$$
d_n(V^{m,k}_{r,l}, \, l^{m,k}_{q,\sigma}) \stackrel{(\ref{dn_vmk_22})}{\gtrsim} m^{\frac 1q‐\frac 12} k^{\frac{1}{\sigma}‐\frac 12}r^{\frac 12} l^{\frac 12}, \quad n\le \frac{mk}{2}.
$$
Hence, for $a(q, \, \sigma)mk\le n\le \frac{mk}{2}$ we have
$$
d_n(V^{m,k}_{r,l}, \, l^{m,k}_{q,\sigma}) \underset{q,\sigma}{\gtrsim} n^{‐\frac 12}m^{\frac 1q} k^{\frac{1}{\sigma}}r^{\frac 12} l^{\frac 12}. 
$$
This completes the proof.
\end{proof}

\begin{Cor}
\label{cor1} Let $2\le q<\infty$, $2\le \sigma <\infty$, $1\le p\le q$, $1\le \theta \le \sigma$, $m$, $k$, $n\in \N$, $a(q, \, \sigma)mk \le n \le \frac{mk}{2}$. Then
$$
d_n(B^{m,k}_{p,\theta}, \, l^{m,k}_{q,\sigma}) \underset{q,\sigma}{\gtrsim} \begin{cases} m^{\frac 1q‐\frac 1p} k^{\frac{1}{\sigma} ‐\frac{1}{\theta}} & \text{if } \min\{p,\, \theta\}\ge 2, \\ m^{\frac 1q ‐\frac 1p +\frac 12}n^{‐\frac 12}k^{\frac{1}{\sigma}} & \text{if } p\ge 2, \; \theta \le 2, \\ k^{\frac{1}{\sigma} ‐\frac{1}{\theta} +\frac 12}n^{‐\frac 12}m^{\frac{1}{q}}, & \text{if } \theta\ge 2, \; p \le 2, \\ n^{‐\frac 12}m^{\frac{1}{q}}k^{\frac{1}{\sigma}} & \text{if }\max\{p, \, \theta\}\le 2.\end{cases}
$$
\end{Cor}
\begin{proof}
For $\min\{p, \, \theta\}\ge 2$, we use the inclusion $m^{‐\frac 1p}k^{‐\frac{1}{\theta}}V^{m,k}_{m,k} \subset B^{m,k}_{p,\theta}$, for $\theta \le 2\le p$, the inclusion $m^{‐\frac 1p}V^{m,k}_{m,1}\subset B^{m,k}_{p,\theta}$, for $p\le 2 \le \theta$, the inclusion $k^{‐\frac{1}{\theta}}V^{m,k}_{1,k} \subset B^{m,k}_{p,\theta}$, for $\max\{p, \, \theta\}\le 2$, the inclusion $V^{m,k}_{1,1}\subset B^{m,k}_{p,\theta}$.
\end{proof}

In what follows, $m$, $k\in \N$, $n\in \Z_+$, $n\le \frac{mk}{2}$, $\nu_i>0$, $1=1, \, 2$.

\begin{Lem}
\label{emb}
Let $1\le p_i\le \infty$, $1\le \theta_i\le \infty$, $\lambda \in [0, \, 1]$. We define the numbers $p$, $\theta\in [1, \, \infty]$ by
\begin{align}
\label{emb_pt}
\frac 1p = \frac{1‐\lambda}{p_1} + \frac{\lambda}{p_2}, \quad \frac{1}{\theta} = \frac{1‐\lambda}{\theta_1} + \frac{\lambda}{\theta_2}.
\end{align}
Then
$$
\nu_1B_{p_1,\theta_1}^{m,k} \cap \nu_2B_{p_2,\theta_2} ^{m,k} \subset \nu_1^{1‐\lambda} \nu_2^\lambda B^{m,k} _{p,\theta}.
$$
\end{Lem}
\begin{proof}
It is sufficient to prove that
$$
\|(x_{i,j})_{1\le i\le m,\, 1\le j\le k}\|_{l_{p,\theta}^{m,k}} \le \|(x_{i,j})_{1\le i\le m,\, 1\le j\le k}\|^{1‐\lambda}_{l_{p_1,\theta_1}^{m,k}} \|(x_{i,j})_{1\le i\le m,\, 1\le j\le k}\|_{l_{p_2,\theta_2}^{m,k}}^\lambda.
$$
We define the number $\beta$ by the equation $\frac{\beta}{p} = \frac{\lambda}{p_2}$. From (\ref{emb_pt}) it follows that $\frac{1‐\beta}{p} = \frac{1‐\lambda}{p_1}$; hence $\beta \in [0, \, 1]$. By H\"{o}lder's inequality, we get for each $j\in \{1, \, \dots, \, k\}$
\begin{align}
\label{xij_p_k}
\|(x_{i,j})_{1\le i\le m}\|_{l_p^m} =\left(\sum \limits _{i=1}^m |x_{i,j}|^{p(1‐\lambda)} |x_{i,j}|^{p\lambda}\right)^{\frac 1p} \le \|(x_{i,j})_{1\le i\le m}\|_{l_{p_1}^m}^{1‐\lambda}\|(x_{i,j})_{1\le i\le m}\|_{l_{p_2}^m}^{\lambda}.
\end{align}
Now we define $\gamma$ by the equation $\frac{\gamma}{\theta} = \frac{\lambda}{\theta_2}$. Then $\frac{1-\gamma}{\theta} \stackrel{(\ref{emb_pt})}{=} \frac{1-\lambda}{\theta_1}$; hence $\gamma \in [0, \, 1]$. By H\"{o}lder's inequality, we obtain
$$
\|(x_{i,j})_{1\le i\le m,\, 1\le j\le k}\|_{l_{p,\theta}^{m,k}} \stackrel{(\ref{xij_p_k})}{\le} \left(\sum \limits _{j=1}^k \|(x_{i,j})_{1\le i\le m}\|_{l_{p_1}^m}^{(1‐\lambda)\theta}\|(x_{i,j})_{1\le i\le m}\|_{l_{p_2}^m}^{\lambda\theta}\right)^{\frac{1}{\theta}}\le
$$
$$
\le \|(x_{i,j})_{1\le i\le m,\, 1\le j\le k}\|^{1‐\lambda}_{l_{p_1,\theta_1}^{m,k}} \|(x_{i,j})_{1\le i\le m,\, 1\le j\le k}\|_{l_{p_2,\theta_2}^{m,k}}^\lambda.
$$
This completes the proof.
\end{proof}

\begin{Lem}
\label{incl_1} Let $\lambda \in [0, \, 1]$, $\frac 1p = \frac{1‐\lambda}{p_1} +\frac{\lambda}{p_2}$, $\frac{1}{\theta} = \frac{1‐\lambda}{\theta_1} +\frac{\lambda}{\theta_2}$, $\tilde r \in [1, \, m]$, $\tilde l = [1, \, k]$, $r = \lfloor \tilde r \rfloor$ or $r=\lceil \tilde r \rceil$, $l = \lfloor \tilde l \rfloor$ or $l=\lceil \tilde l \rceil$, 
\begin{align}
\label{n1n2rp1lt1}
\frac{\nu_1}{\nu_2} = \tilde r^{\frac{1}{p_1}‐\frac{1}{p_2}} \tilde l ^{\frac{1}{\theta_1} ‐\frac{1}{\theta_2}}.
\end{align}
Then
$$
\nu_1^{1‐\lambda} \nu_2^\lambda r^{‐\frac 1p} l^{‐\frac{1}{\theta}} V^{m,k}_{r,l} \subset 4(\nu_1B_{p_1,\theta_1}^{m,k}\cap \nu_2 B_{p_2,\theta_2}^{m,k}).
$$
\end{Lem}
\begin{proof}
By (\ref{gamma_x_def})‐‐(\ref{vrl_km}), it suffices to prove that $$\nu_1^{1‐\lambda}\nu_2^{\lambda} \tilde r^{\frac{1}{p_1} ‐\frac 1p} \tilde l ^{\frac{1}{\theta_1}‐\frac{1}{\theta}}\le \nu_1, \;\;\nu_1^{1‐\lambda}\nu_2^{\lambda} \tilde r^{\frac{1}{p_2}‐\frac 1p} \tilde l ^{\frac{1}{\theta_2}‐\frac{1}{\theta}}\le \nu_2.$$ It follows from (\ref{n1n2rp1lt1}).
\end{proof}

\begin{Lem}
\label{lemma01} Let $2\le q<\infty$, $2\le \sigma<\infty$. Then
$$
d_n(\nu_1B^{m,k}_{p_1,\theta_1}\cap \nu_2B^{m,k}_{p_2,\theta_2}, \, l^{m,k}_{q,\sigma}) \underset{q,\sigma}{\gtrsim} \min\{\nu_1, \, \nu_2\}\min\{1, \, n^{‐\frac 12}m^{\frac 1q}k^{\frac{1}{\sigma}}\}.
$$
\end{Lem}
\begin{proof}
From the inclusion $\min\{\nu_1, \, \nu_2\} V^{m,k}_{1,1} \subset \nu_1B^{m,k}_{p_1,\theta_1}\cap \nu_2B^{m,k}_{p_2,\theta_2}$ and Proposition \ref{v_dn} it follows that
$$
d_n(\nu_1B^{m,k}_{p_1,\theta_1}\cap \nu_2B^{m,k}_{p_2,\theta_2}, \, l^{m,k}_{q,\sigma})\ge 
$$
$$
\ge d_n(\min\{\nu_1, \, \nu_2\} V^{m,k}_{1,1}, \, l^{m,k}_{q,\sigma}) \underset{q,\sigma}{\gtrsim} \min\{\nu_1, \, \nu_2\}\min\{1, \, n^{‐\frac 12}m^{\frac 1q}k^{\frac{1}{\sigma}}\}.
$$
This completes the proof.
\end{proof}

\begin{Lem}
\label{lemma02}
Let $2\le q<\infty$, $2\le \sigma <\infty$.
\begin{enumerate}
\item Let $q>2$, $m^{\frac 2q} k^{\frac{2}{\sigma}}\le n\le mk^{\frac{2}{\sigma}}$, 
\begin{align}
\label{min1}
\frac{\nu_1}{\nu_2} \le (n^{\frac 12}m^{‐\frac 1q} k^{‐\frac{1}{\sigma}})^{\frac{1/p_1‐1/p_2}{1/2‐1/q}}.
\end{align}
Then
\begin{align}
\label{dn_nu1}
d_n(\nu_1B^{m,k}_{p_1,\theta_1}\cap \nu_2B^{m,k}_{p_2,\theta_2}, \, l^{m,k}_{q,\sigma}) \underset{q,\sigma}{\gtrsim} \nu_1(n^{‐\frac 12} m^{\frac 1q} k^{\frac{1}{\sigma}}) ^{\frac{1/p_1‐1/q}{1/2‐1/q}}.
\end{align}
\item Let $\sigma>2$, $m^{\frac 2q} k^{\frac{2}{\sigma}}\le n\le m^{\frac 2q}k$, 
$$
\frac{\nu_1}{\nu_2} \le (n^{\frac 12}m^{‐\frac 1q} k^{‐\frac{1}{\sigma}})^{\frac{1/\theta_1‐1/\theta_2}{1/2‐1/\sigma}}.
$$
Then
$$
d_n(\nu_1B^{m,k}_{p_1,\theta_1}\cap \nu_2B^{m,k}_{p_2,\theta_2}, \, l^{m,k}_{q,\sigma}) \underset{q,\sigma}{\gtrsim} \nu_1(n^{‐\frac 12} m^{\frac 1q} k^{\frac{1}{\sigma}}) ^{\frac{1/\theta_1‐1/\sigma}{1/2‐1/\sigma}}.
$$
\end{enumerate}
\end{Lem}

\begin{proof}
We prove assertion 1 (assertion 2 is similar).

We set
\begin{align}
\label{02rdef}
r = \left\lceil(n^{\frac 12}m^{‐\frac 1q}k^{‐\frac{1}{\sigma}})^{\frac{1}{\frac 12‐\frac 1q}}\right\rceil.
\end{align}
Since $m^{\frac 2q} k^{\frac{2}{\sigma}}\le n\le mk^{\frac{2}{\sigma}}$, we have $1\le r\le m$. We claim that
\begin{align}
\label{vkl1}
\nu_1 r^{‐\frac{1}{p_1}} V^{m,k}_{r,1} \subset 2(\nu_1B^{m,k}_{p_1,\theta_1}\cap \nu_2B^{m,k}_{p_2,\theta_2}).
\end{align}
It suffices to check that $\nu_1 r^{\frac{1}{p_2}‐\frac{1}{p_1}}\le 2\nu_2$ (see (\ref{eij_kl}), (\ref{vrl_km})). It follows from (\ref{min1}) and (\ref{02rdef}).

By (\ref{02rdef}), we have $n\le m^{\frac 2q}k^{\frac{2}{\sigma}} r^{1‐\frac{2}{q}}$. Hence
$$
d_n(\nu_1B^{m,k}_{p_1,\theta_1}\cap \nu_2B^{m,k}_{p_2,\theta_2}, \, l^{m,k}_{q,\sigma}) \stackrel{(\ref{vkl1})}{\gtrsim} \nu_1 r^{‐\frac{1}{p_1}}d_n(V_{r,1}^{m,k}, \, l_{q,\sigma}^{m,k}) 
\stackrel{(\ref{dn_vmk1})}{\underset{q,\sigma}{\gtrsim}} \nu_1 r^{\frac 1q‐\frac{1}{p_1}};
$$
this together with (\ref{02rdef}) implies (\ref{dn_nu1}).
\end{proof}

\begin{Lem}
\label{lemma03} Let $2\le q<\infty$, $2\le \sigma <\infty$.
\begin{enumerate}
\item Let $\sigma>2$, $mk^{\frac{2}{\sigma}} \le n \le \frac{mk}{2}$, 
\begin{align}
\label{min2} \frac{\nu_1}{\nu_2} \le m^{\frac{1}{p_1} ‐\frac{1}{p_2}} (n^{\frac 12}m^{‐\frac 12} k^{‐\frac{1}{\sigma}}) ^{\frac{1/\theta_1‐1/\theta_2}{1/2‐1/\sigma}}.
\end{align}
Then
\begin{align}
\label{dn_nu11}
d_n(\nu_1B^{m,k}_{p_1,\theta_1}\cap \nu_2B^{m,k}_{p_2,\theta_2}, \, l^{m,k}_{q,\sigma}) \underset{q,\sigma}{\gtrsim} \nu_1 m^{\frac 1q‐\frac{1}{p_1}}(n^{‐\frac 12} m^{\frac 12} k^{\frac{1}{\sigma}}) ^{\frac{1/\theta_1‐1/\sigma}{1/2‐1/\sigma}}.
\end{align}

\item Let $q>2$, $m^{\frac 2q}k\le n \le \frac{mk}{2}$, 
$$
\frac{\nu_1}{\nu_2} \le k^{\frac{1}{\theta_1} ‐\frac{1}{\theta_2}} (n^{\frac 12}m^{‐\frac 1q} k^{‐\frac{1}{2}}) ^{\frac{1/p_1‐1/p_2}{1/2‐1/q}}.
$$
Then
$$
d_n(\nu_1B^{m,k}_{p_1,\theta_1}\cap \nu_2B^{m,k}_{p_2,\theta_2}, \, l^{m,k}_{q,\sigma}) \underset{q,\sigma}{\gtrsim} \nu_1 k^{\frac{1}{\sigma}‐\frac{1}{\theta_1}}(n^{‐\frac 12} m^{\frac 1q} k^{\frac{1}{2}}) ^{\frac{1/p_1‐1/q}{1/2‐1/q}}.
$$
\end{enumerate}
\end{Lem}

\begin{proof}
We prove assertion 1 (assertion 2 is similar).

Let 
\begin{align}
\label{03ldef}
l =\left\lceil \left(n^{\frac 12}m^{‐\frac 12}k^{‐\frac{1}{\sigma}}\right)^{\frac{1}{\frac 12‐\frac{1}{\sigma}}}\right\rceil.
\end{align}
Since $mk^{2/\sigma}\le n \le \frac{mk}{2}$, we have $1\le l\le k$.

We prove that
\begin{align}
\label{vkl2} \nu_1 m^{‐\frac{1}{p_1}} l^{‐\frac{1}{\theta_1}} V^{m,k}_{m,l} \subset 2(\nu_1B^{m,k}_{p_1,\theta_1}\cap \nu_2B^{m,k}_{p_2,\theta_2}).
\end{align}
It suffices to check that
$$
\nu_1 m^{\frac{1}{p_2}‐\frac{1}{p_1}} l^{\frac{1}{\theta_2}‐\frac{1}{\theta_1}} \le 2\nu_2.
$$
It follows from (\ref{min2}) and (\ref{03ldef}).

By (\ref{03ldef}), we have $n\le m^{\frac 2q}k^{\frac{2}{\sigma}} m^{1‐\frac 2q}l^{1‐\frac{2}{\sigma}}$.

Now we obtain
$$
d_n(\nu_1B^{m,k}_{p_1,\theta_1}\cap \nu_2B^{m,k}_{p_2,\theta_2}, \, l^{m,k}_{q,\sigma}) \stackrel{(\ref{vkl2})}{\gtrsim} \nu_1 m^{‐\frac{1}{p_1}} l^{‐\frac{1}{\theta_1}}d_n(V_{m,l}^{m,k}, \, l^{m,k}_{q,\sigma})\stackrel{(\ref{dn_vmk1})}{\underset{q,\sigma}{\gtrsim}} \nu_1m^{\frac 1q‐\frac{1}{p_1}}l^{\frac{1}{\sigma}‐\frac{1}{\theta_1}}.
$$
This together with (\ref{03ldef}) implies (\ref{dn_nu11}).
\end{proof}

\begin{Lem}
\label{lemma04} Let $q>2$, $\sigma>2$,
\begin{align}
\label{n_mk_mk} m^{2/q}k^{2/\sigma} \le n\le \min\{mk^{2/\sigma}, \, m^{2/q}k\},
\end{align}
\begin{align}
\label{nu1nu2int} (n^{\frac 12} m^{‐\frac 1q} k^{‐\frac{1}{\sigma}})^{\frac{1/p_1‐1/p_2} {1/2‐1/q}} \le \frac{\nu_1}{\nu_2} \le (n^{\frac 12} m^{‐\frac 1q} k^{‐\frac{1}{\sigma}})^{\frac{1/\theta_1‐1/\theta_2}{1/2‐1/\sigma}}
\end{align}
or
\begin{align}
\label{nu1nu2int1} (n^{\frac 12} m^{‐\frac 1q} k^{‐\frac{1}{\sigma}})^{\frac{1/\theta_1‐1/\theta_2}{1/2‐1/\sigma}} \le \frac{\nu_1}{\nu_2} \le (n^{\frac 12} m^{‐\frac 1q} k^{‐\frac{1}{\sigma}})^{\frac{1/p_1‐1/p_2} {1/2‐1/q}},
\end{align}
$\lambda \in [0, \, 1]$, $\frac 1p = \frac{1‐\lambda}{p_1} + \frac{\lambda}{p_2}$, $\frac{1}{\theta} = \frac{1‐\lambda}{\theta_1} + \frac{\lambda}{\theta_2}$, 
\begin{align}
\label{pq_eq_ts}
\frac{1/p‐1/q}{1/2‐1/q} = \frac{1/\theta ‐1/\sigma}{1/2 ‐1/\sigma}.
\end{align}
Then
\begin{align}
\label{p_th}
d_n(\nu_1B^{m,k}_{p_1,\theta_1} \cap \nu_2 B_{p_2,\theta_2}^{m,k}, \, l^{m,k}_{q,\sigma}) \underset{q,\sigma}{\gtrsim} \nu_1^{1‐\lambda}\nu_2^\lambda (n^{‐\frac 12} m^{\frac 1q} k^{\frac{1}{\sigma}})^{\frac{1/p‐1/q}{1/2‐1/q}} \stackrel{(\ref{1234})}{=} \nu_1^{1‐\lambda}\nu_2^\lambda (n^{‐\frac 12} m^{\frac 1q} k^{\frac{1}{\sigma}})^{\frac{1/\theta‐1/\sigma}{1/2‐1/\sigma}}.
\end{align}
\end{Lem}
\begin{proof}
We set
\begin{align}
\label{tiltalphtilr}
\tilde r = (n^{\frac 12}m^{‐\frac 1q} k^{‐\frac{1}{\sigma}})^{\frac{1‐\alpha}{1/2‐1/q}}, \quad \tilde l = (n^{\frac 12}m^{‐\frac 1q} k^{‐\frac{1}{\sigma}})^{\frac{\alpha}{1/2‐1/\sigma}},
\end{align}
where $\alpha \in [0, \, 1]$. By (\ref{n_mk_mk}), we have $\tilde r \in [1, \, m]$, $\tilde l\in [1, \, k]$. We choose $\alpha$ such that (\ref{n1n2rp1lt1}) holds; it exists by (\ref{nu1nu2int}) or (\ref{nu1nu2int1}).

Let $r= \lceil \tilde r\rceil$, $l =\lceil \tilde l \rceil$, $W=\nu_1^{1‐\lambda} \nu_2^\lambda r^{‐\frac{1}{p}}l^{‐\frac{1}{\theta}} V_{r,l}^{m,k}$. By Lemma \ref{incl_1},
\begin{align}
\label{w_sbst_4}
W \subset 4(\nu_1 B^{m,k}_{p_1,\theta_1} \cap \nu_2 B^{m,k}_{p_2,\theta_2}).
\end{align}

From (\ref{tiltalphtilr}) it follows that $n\le m^{\frac 2q} k^{\frac{2}{\sigma}} r^{1‐\frac 2q} l^{1‐\frac{2}{\sigma}}$. Hence
$$
d_n(\nu_1B^{m,k}_{p_1,\theta_1} \cap \nu_2 B_{p_2,\theta_2}^{m,k}, \, l^{m,k}_{q,\sigma}) \stackrel{(\ref{w_sbst_4})}{\gtrsim} d_n(W, \, l^{m,k}_{q,\sigma}) \stackrel{(\ref{dn_vmk1})}{\underset{q,\sigma}{\gtrsim}} \nu_1^{1‐\lambda} \nu_2^\lambda r^{1/q‐1/p} l^{1/\sigma ‐1/\theta}.
$$
This together with (\ref{1234}), (\ref{pq_eq_ts}), (\ref{tiltalphtilr}) yields (\ref{p_th}).
\end{proof}

\begin{Lem}
\label{lemma05} Let $q>2$, $\sigma>2$, $1\le p_i\le \infty$, $1\le \theta_i\le \infty$, $i=1, \, 2$, $\lambda \in [0, \, 1]$, $\frac 1p=\frac{1‐\lambda}{p_1} + \frac{\lambda}{p_2}$, $\frac{1}{\theta} = \frac{1‐\lambda}{\theta_1} + \frac{\lambda}{\theta_2}$, 
\begin{align}
\label{lpq_lts_eq}
\frac{1/p‐1/q}{1/2‐1/q} = \frac{1/\theta‐1/\sigma}{1/2‐1/\sigma}. 
\end{align}
Let one of the following conditions hold:
\begin{enumerate}
\item $km^{\frac 2q}\le n \le mk^{\frac{2}{\sigma}}$, 
\begin{align}
\label{nu1nu2lege1} \begin{array}{c}(n^{\frac 12} m^{‐\frac 1q} k^{‐\frac{1}{\sigma}}) ^{\frac{1/p_1‐1/p_2}{1/2‐1/q}}\le \frac{\nu_1}{\nu_2} \le k^{\frac{1}{\theta_1}‐\frac{1}{\theta_2}} (n^{\frac 12} m^{‐\frac 1q} k^{‐\frac{1}{2}}) ^{\frac{1/p_1‐1/p_2}{1/2‐1/q}} \quad \text{or} \\  k^{\frac{1}{\theta_1}‐\frac{1}{\theta_2}} (n^{\frac 12} m^{‐\frac 1q} k^{‐\frac{1}{2}}) ^{\frac{1/p_1‐1/p_2}{1/2‐1/q}} \le \frac{\nu_1}{\nu_2} \le (n^{\frac 12} m^{‐\frac 1q} k^{‐\frac{1}{\sigma}}) ^{\frac{1/p_1‐1/p_2}{1/2‐1/q}}; \end{array}
\end{align}
\item $mk^{\frac{2}{\sigma}}\le n \le km^{\frac 2q}$, 
$$ \begin{array}{c} (n^{\frac 12} m^{‐\frac 1q} k^{‐\frac{1}{\sigma}}) ^{\frac{1/\theta_1‐1/\theta_2}{1/2‐1/\sigma}}\le \frac{\nu_1}{\nu_2} \le m^{\frac{1}{p_1}‐\frac{1}{p_2}} (n^{\frac 12} m^{‐\frac{1}{2}} k^{‐\frac{1}{\sigma}}) ^{\frac{1/\theta_1‐1/\theta_2}{1/2‐1/\sigma}} \quad \text{or} \\ m^{\frac{1}{p_1}‐\frac{1}{p_2}} (n^{\frac 12} m^{‐\frac{1}{2}} k^{‐\frac{1}{\sigma}}) ^{\frac{1/\theta_1‐1/\theta_2}{1/2‐1/\sigma}} \le \frac{\nu_1}{\nu_2} \le (n^{\frac 12} m^{‐\frac 1q} k^{‐\frac{1}{\sigma}}) ^{\frac{1/\theta_1‐1/\theta_2}{1/2‐1/\sigma}}.\end{array}
$$
\end{enumerate}
Then
$$
d_n(\nu_1B^{m,k}_{p_1,\theta_1} \cap \nu_2 B^{m,k}_{p_2,\theta_2}, \, l^{m,k}_{q,\sigma}) \underset{q,\sigma}{\gtrsim} \nu_1^{1‐\lambda}\nu_2^\lambda (n^{‐\frac 12} m^{\frac 1q} k^{\frac{1}{\sigma}})^{\frac{1/p‐1/q}{1/2‐1/q}} \stackrel{(\ref{1234})}{=}\nu_1^{1‐\lambda}\nu_2^\lambda (n^{‐\frac 12} m^{\frac 1q} k^{\frac{1}{\sigma}})^{\frac{1/\theta‐1/\sigma}{1/2‐1/\sigma}}.
$$
\end{Lem}
\begin{proof}
Let condition 1 hold (the case of condition 2 is similar).

We set $W=\nu_1^{1‐\lambda} \nu_2^\lambda r^{‐\frac 1p} l^{‐\frac{1}{\theta}} V_{r,l}^{m,k}$, where $r=\lceil \tilde r\rceil$, $l = \lceil \tilde l\rceil$,
\begin{align}
\label{tilrdef_alph}
\tilde r=(n^\frac 12 m^{‐\frac 1q}k^{‐\frac{1}{\sigma}})^{\frac{1‐\alpha}{1/2‐1/q}}(n^\frac 12 m^{‐\frac 1q}k^{‐\frac{1}{2}})^{\frac{\alpha}{1/2‐1/q}}, \quad \tilde l = k^\alpha,
\end{align}
$\alpha \in [0, \, 1]$. Since $km^{\frac 2q}\le n \le mk^{\frac{2}{\sigma}}$, we have $1\le r\le m$, $1\le l\le k$.

From (\ref{nu1nu2lege1}) it follows that there exists $\alpha \in [0, \, 1]$ such that (\ref{n1n2rp1lt1}) holds.

By Lemma \ref{incl_1}, 
\begin{align}
\label{wsb4}
W \subset 4(\nu_1B^{m,k}_{p_1,\theta_1} \cap \nu_2 B^{m,k}_{p_2,\theta_2}).
\end{align}

From (\ref{tilrdef_alph}) it follows that $n\le m^{\frac 2q}k^{\frac{2}{\sigma}}r^{1‐\frac 2q} l^{1‐\frac{2}{\sigma}}$; therefore,
$$
d_n(\nu_1B^{m,k}_{p_1,\theta_1} \cap \nu_2 B^{m,k}_{p_2,\theta_2}, \, l^{m,k}_{q,\sigma})\stackrel{(\ref{wsb4})}{\gtrsim} d_n(\nu_1^{1‐\lambda} \nu_2^\lambda r^{‐\frac 1p} l^{‐\frac{1}{\theta}} V_{r,l}^{m,k}, \, l^{m,k}_{q,\sigma}) \stackrel{(\ref{dn_vmk1})}{\underset{q,\sigma}{\gtrsim}}
$$
$$
\gtrsim \nu_1^{1‐\lambda} \nu_2^\lambda r^{\frac 1q‐\frac 1p} l^{\frac{1}{\sigma}‐\frac{1}{\theta}};
$$
this together with (\ref{1234}), (\ref{lpq_lts_eq}), (\ref{tilrdef_alph}) implies the desired estimate.
\end{proof}

\begin{Lem}
\label{lemma06} Let $q>2$, $\sigma>2$, $\max\{mk^{\frac{2}{\sigma}}, \, m^{\frac 2q}k\}\le n \le \frac{mk}{2}$, $\lambda \in [0, \, 1]$, $\frac 1p = \frac{1‐\lambda}{p_1} +\frac{\lambda}{p_2}$, $\frac{1}{\theta} = \frac{1‐\lambda}{\theta_1} +\frac{\lambda}{\theta_2}$, 
\begin{align}
\label{ttt}
\frac{1/p‐1/q}{1/2‐1/q}=\frac{1/\theta‐1/\sigma}{1/2‐1/\sigma}. 
\end{align}
Let one of the following conditions hold:
\begin{align}
\label{nu1nu2in5} m^{\frac{1}{p_1}‐\frac{1}{p_2}} (n^{\frac 12} m^{‐\frac 12} k^{‐\frac{1}{\sigma}})^{\frac{1/\theta_1‐1/\theta_2}{1/2‐1/\sigma}} \le \frac{\nu_1}{\nu_2} \le k^{\frac{1}{\theta_1} ‐\frac{1}{\theta_2}}(n^{\frac 12} m^{‐\frac 1q} k^{‐\frac{1}{2}}) ^{\frac{1/p_1‐1/p_2}{1/2‐1/q}}
\end{align}
or
\begin{align}
\label{nu1nu2in500} k^{\frac{1}{\theta_1} ‐\frac{1}{\theta_2}}(n^{\frac 12} m^{‐\frac 1q} k^{‐\frac{1}{2}}) ^{\frac{1/p_1‐1/p_2}{1/2‐1/q}} \le \frac{\nu_1}{\nu_2} \le m^{\frac{1}{p_1}‐\frac{1}{p_2}} (n^{\frac 12} m^{‐\frac 12} k^{‐\frac{1}{\sigma}})^{\frac{1/\theta_1‐1/\theta_2}{1/2‐1/\sigma}}.
\end{align}
Then
$$
d_n(\nu_1B^{m,k}_{p_1,\theta_1}\cap\nu_2B^{m,k}_{p_2,\theta_2}, \, l^{m,k}_{q,\sigma}) \underset{q,\sigma}{\gtrsim} \nu_1^{1‐\lambda}\nu_2^\lambda k^{\frac{1}{\sigma}‐ \frac{1}{\theta}}(n^{‐\frac 12} m^{\frac 1q} k^{\frac{1}{2}})^{\frac{1/p‐1/q}{1/2‐1/q}} \stackrel{(\ref{1234})}{=}
$$
$$
=\nu_1^{1‐\lambda}\nu_2^\lambda m^{\frac{1}{q}‐ \frac{1}{p}}(n^{‐\frac 12} m^{\frac 12} k^{\frac{1}{\sigma}})^{\frac{1/\theta‐1/\sigma}{1/2‐1/\sigma}}.
$$
\end{Lem}
\begin{proof}
We set $W\subset \nu_1^{1‐\lambda}\nu_2^\lambda r^{‐\frac 1p} l^{‐\frac{1}{\theta}} V_{r,l}^{m,k}$, where $r=\lceil \tilde r\rceil$, $l= \lceil \tilde l \rceil$, 
\begin{align}
\label{rmn12}
\tilde r = m^{1‐\alpha}(n^{\frac 12}m^{‐\frac 1q}k^{‐\frac 12})^{\frac{\alpha}{1/2‐1/q}}, \quad \tilde l = (n^{\frac 12} m^{‐\frac 12} k^{‐\frac{1}{\sigma}})^{\frac{1‐\alpha}{1/2‐1/\sigma}}k^\alpha,
\end{align}
$\alpha \in [0, \, 1]$. Since $\max\{mk^{\frac{2}{\sigma}}, \, m^{\frac 2q}k\}\le n \le \frac{mk}{2}$, we get $1\le r\le m$, $1\le l\le k$.

By (\ref{nu1nu2in5}) or (\ref{nu1nu2in500}), there is $\alpha \in [0, \, 1]$ such that (\ref{n1n2rp1lt1}) holds. From Lemma \ref{incl_1} it follows that 
\begin{align}
\label{w_s_4_n}
W\subset 4(\nu_1B^{m,k}_{p_1,\theta_1}\cap\nu_2B^{m,k}_{p_2,\theta_2}). 
\end{align}
By (\ref{rmn12}), $n\le m^{\frac 2q}k^{\frac{2}{\sigma}} r^{1‐\frac 2q} l^{1‐\frac{2}{\sigma}}$. Hence
$$
d_n(\nu_1B^{m,k}_{p_1,\theta_1}\cap\nu_2B^{m,k}_{p_2,\theta_2}, \, l^{m,k}_{q,\sigma}) \stackrel{(\ref{w_s_4_n})}{\gtrsim} d_n(\nu_1^{1‐\lambda}\nu_2^\lambda r^{‐\frac 1p} l^{‐\frac{1}{\theta}} V_{r,l}^{m,k}, \, l^{m,k}_{q,\sigma}) \stackrel{(\ref{dn_vmk1})}{\underset{q,\sigma}{\gtrsim}}
$$
$$
\gtrsim \nu_1^{1‐\lambda}\nu_2^\lambda r^{\frac 1q‐\frac 1p} l^{\frac{1}{\sigma}‐\frac{1}{\theta}};
$$
this together with (\ref{1234}), (\ref{ttt}), (\ref{rmn12}) yields the desired estimate.
\end{proof}

\begin{Lem}
\label{lemma07} 
Let $q\ge 2$, $\sigma\ge 2$.
\begin{enumerate}
\item Let $mk^{2/\sigma}\le n\le \frac{mk}{2}$, $\frac{\nu_1}{\nu_2} \ge m^{\frac{1}{p_1}‐\frac{1}{p_2}}$.
Then
$$
d_n(\nu_1B_{p_1,\theta_1}^{m,k}\cap \nu_2 B_{p_2,\theta_2}^{m,k}, \, l_{q,\sigma}^{m,k}) \underset{q,\sigma}{\gtrsim} \nu_2 m^{\frac 1q‐\frac{1}{p_2}+\frac 12} k^{\frac{1}{\sigma}} n^{‐\frac 12}.
$$

\item Let $m^{2/q}k\le n \le \frac{mk}{2}$, $\frac{\nu_1}{\nu_2} \ge k^{\frac{1}{\theta_1}‐\frac{1}{\theta_2}}$.
Then
$$
d_n(\nu_1B_{p_1,\theta_1}^{m,k}\cap \nu_2 B_{p_2,\theta_2}^{m,k}, \, l_{q,\sigma}^{m,k}) \underset{q,\sigma}{\gtrsim} \nu_2 k^{\frac{1}{\sigma}‐\frac{1}{\theta_2}+\frac 12} m^{\frac{1}{q}} n^{‐\frac 12}.
$$
\end{enumerate}
\end{Lem}
\begin{proof}
We consider assertion 1 (assertion 2 is similar).

We set $W=\nu_2 m^{‐\frac{1}{p_2}}V^{m,k}_{m,1}$. From the inequality $\nu_2 m^{\frac{1}{p_1}‐\frac{1}{p_2}}\le \nu_1$ it follows that $W \subset \nu_1B_{p_1,\theta_1}^{m,k}\cap \nu_2 B_{p_2,\theta_2}^{m,k}$. We have $n\ge m^{\frac 2q}k^{\frac{2}{\sigma}} m^{1‐\frac 2q}$. Hence
$$
d_n(\nu_1B_{p_1,\theta_1}^{m,k}\cap \nu_2 B_{p_2,\theta_2}^{m,k}, \, l_{q,\sigma}^{m,k}) \ge d_n(\nu_2 m^{‐\frac{1}{p_2}}V_{m,1}^{m,k}, \, l_{q,\sigma}^{m,k}) \stackrel{(\ref{dn_vmk1})}{\underset{q,\sigma}{\gtrsim}} \nu_2 m^{‐\frac{1}{p_2}} n^{‐\frac 12} m^{\frac 1q} k^{\frac{1}{\sigma}}m^{\frac 12}.
$$
This completes the proof.
\end{proof}

\begin{Lem}
\label{lemma08} 
Let $q\ge 2$, $\sigma\ge 2$.
\begin{enumerate}
\item Let $mk^{2/\sigma}\le n \le \frac{mk}{2}$, $\tilde \mu\in [0, \, 1]$, $\frac 12 = \frac{1‐\tilde \mu}{\theta_1} +\frac{\tilde \mu}{\theta_2}$, $\frac{1}{\tilde p} = \frac{1‐\tilde \mu}{p_1} +\frac{\tilde \mu}{p_2}$,
$$
m^{\frac{1}{p_1}‐\frac{1}{p_2}} (n^{\frac 12}m^{‐\frac 12}k^{‐\frac{1}{\sigma}})^{\frac{1/\theta_1‐1/\theta_2}{1/2‐1/\sigma}} \le \frac{\nu_1}{\nu_2} \le m^{\frac{1}{p_1}‐\frac{1}{p_2}}
$$
or
$$
m^{\frac{1}{p_1}‐\frac{1}{p_2}} \le \frac{\nu_1}{\nu_2} \le m^{\frac{1}{p_1}‐\frac{1}{p_2}} (n^{\frac 12}m^{‐\frac 12}k^{‐\frac{1}{\sigma}})^{\frac{1/\theta_1‐1/\theta_2}{1/2‐1/\sigma}}.
$$
Then
$$
d_n(\nu_1B_{p_1,\theta_1}^{m,k}\cap \nu_2 B_{p_2,\theta_2}^{m,k}, \, l_{q,\sigma}^{m,k}) \underset{q,\sigma}{\gtrsim} \nu_1^{1‐\tilde \mu}\nu_2^{\tilde \mu} m^{\frac 1q‐\frac{1}{\tilde p}+\frac 12} k^{\frac{1}{\sigma}} n^{‐\frac 12}.
$$

\item Let $m^{2/q}k\le n \le \frac{mk}{2}$, $\tilde \lambda\in [0, \, 1]$, $\frac 12 = \frac{1‐\tilde \lambda}{p_1} +\frac{\tilde \lambda}{p_2}$, $\frac{1}{\tilde \theta} = \frac{1‐\tilde \lambda}{\theta_1} +\frac{\tilde \lambda}{\theta_2}$,
$$
k^{\frac{1}{\theta_1}‐\frac{1}{\theta_2}} (n^{\frac 12}m^{‐\frac 1q}k^{‐\frac{1}{2}})^{\frac{1/p_1‐1/p_2}{1/2‐1/q}} \le \frac{\nu_1}{\nu_2} \le k^{\frac{1}{\theta_1}‐\frac{1}{\theta_2}}
$$
or
$$
k^{\frac{1}{\theta_1}‐\frac{1}{\theta_2}} \le \frac{\nu_1}{\nu_2} \le k^{\frac{1}{\theta_1}‐\frac{1}{\theta_2}} (n^{\frac 12}m^{‐\frac 1q}k^{‐\frac{1}{2}})^{\frac{1/p_1‐1/p_2}{1/2‐1/q}}.
$$
Then
$$
d_n(\nu_1B_{p_1,\theta_1}^{m,k}\cap \nu_2 B_{p_2,\theta_2}^{m,k}, \, l_{q,\sigma}^{m,k}) \underset{q,\sigma}{\gtrsim} \nu_1^{1‐\tilde \lambda}\nu_2^{\tilde \lambda} k^{\frac{1}{\sigma}‐\frac{1}{\tilde \theta}+\frac 12} m^{\frac{1}{q}} n^{‐\frac 12}.
$$
\end{enumerate}
\end{Lem}

\begin{proof}
We prove assertion 1 (assertion 2 is similar).

We set $W=\nu_1^{1‐\tilde \mu}\nu_2^{\tilde \mu} m^{‐\frac{1}{\tilde p}} l^{‐\frac 12}V_{m,l}^{m,k}$, where $l =\lfloor \tilde l\rfloor$, $\tilde l = (n^{\frac 12}m^{‐\frac 12}k^{‐\frac{1}{\sigma}})^{\frac{1‐\alpha}{1/2‐1/\sigma}}$; $\alpha \in [0, \, 1]$ is such that $\frac{\nu_1}{\nu_2} = m^{\frac{1}{p_1}‐\frac{1}{p_2}} \tilde l^{\frac{1}{\theta_1} ‐\frac{1}{\theta_2}}$. Since $mk^{2/\sigma} \le n \le mk$, we have $1\le l\le k$. By Lemma \ref{incl_1}, $W \subset 4(\nu_1B_{p_1,\theta_1}^{m,k}\cap \nu_2 B_{p_2,\theta_2}^{m,k})$. Notice that $n \ge m^{\frac 2q} k^{\frac{2}{\sigma}} m^{1‐\frac 2q} l^{1‐\frac{2}{\sigma}}$. Hence
$$
d_n(\nu_1B_{p_1,\theta_1}^{m,k}\cap \nu_2 B_{p_2,\theta_2}^{m,k}, \, l_{q,\sigma}^{m,k}) \gtrsim d_n(\nu_1^{1‐\tilde \mu}\nu_2^{\tilde \mu} m^{‐\frac{1}{\tilde p}} l^{‐\frac 12}V_{m,l}^{m,k}, \, l_{q,\sigma}^{m,k}) \stackrel{(\ref{dn_vmk1})}{\underset{q,\sigma}{\gtrsim}}
$$
$$
\gtrsim \nu_1^{1‐\tilde \mu}\nu_2^{\tilde \mu} m^{‐\frac{1}{\tilde p}} l^{‐\frac 12} n^{‐\frac 12} m^{\frac 1q} k^{\frac{1}{\sigma}} m^{\frac 12}l^{\frac 12} = \nu_1^{1‐\tilde \mu}\nu_2^{\tilde \mu} m^{\frac 1q‐\frac{1}{\tilde p}} n^{‐\frac 12} m^{\frac 12}k^{\frac{1}{\sigma}}.
$$
This completes the proof.
\end{proof}

\begin{Lem}
\label{lemma09}
Let $q\ge 2$, $\sigma\ge 2$.
\begin{enumerate}
\item Let $m^{2/q}k^{2/\sigma} \le n \le mk^{2/\sigma}$, $1\le \frac{\nu_1}{\nu_2} \le (n^{\frac 12}m^{‐\frac 1q}k^{‐\frac{1}{\sigma}})^{\frac{1/p_1‐1/p_2}{1/2‐1/q}}$ or $(n^{\frac 12}m^{‐\frac 1q}k^{‐\frac{1}{\sigma}})^{\frac{1/p_1‐1/p_2}{1/2‐1/q}} \le \frac{\nu_1}{\nu_2} \le 1$, $\tilde \lambda\in [0, \, 1]$, $\frac{1}{2} = \frac{1‐\tilde \lambda}{p_1} +\frac{\tilde \lambda}{p_2}$. Then 
$$
d_n(\nu_1B_{p_1,\theta_1}^{m,k} \cap \nu_2B_{p_2,\theta_2}^{m,k}, \, l_{q,\sigma}^{m,k}) \underset{q,\sigma}{\gtrsim} \nu_1^{1‐\tilde \lambda} \nu_2^{\tilde \lambda} n^{‐\frac 12}m^{\frac 1q} k^{\frac{1}{\sigma}}.
$$

\item Let $m^{2/q}k^{2/\sigma} \le n \le km^{2/q}$, $1\le \frac{\nu_1}{\nu_2} \le (n^{\frac 12}m^{‐\frac 1q}k^{‐\frac{1}{\sigma}})^{\frac{1/\theta_1‐1/\theta_2}{1/2‐1/\sigma}}$ or $(n^{\frac 12}m^{‐\frac 1q}k^{‐\frac{1}{\sigma}})^{\frac{1/\theta_1‐1/\theta_2}{1/2‐1/\sigma}}\le \frac{\nu_1}{\nu_2} \le 1$, $\tilde \mu\in [0, \, 1]$, $\frac{1}{2} = \frac{1‐\tilde \mu}{\theta_1} +\frac{\tilde \mu}{\theta_2}$. Then
$$
d_n(\nu_1B_{p_1,\theta_1}^{m,k} \cap \nu_2B_{p_2,\theta_2}^{m,k}, \, l_{q,\sigma}^{m,k}) \underset{q,\sigma}{\gtrsim} \nu_1^{1‐\tilde \mu} \nu_2^{\tilde \mu} n^{‐\frac 12}m^{\frac 1q} k^{\frac{1}{\sigma}}.
$$
\end{enumerate}
\end{Lem}
\begin{proof}
We prove assertion 1 (assertion 2 is similar).

 We set $r=\lfloor \tilde r \rfloor$, where $\tilde r = (n^{\frac 12}m^{‐\frac 1q}k^{‐\frac{1}{\sigma}}) ^{\frac{\alpha}{1/2‐1/q}}$, $\alpha \in [0, \, 1]$ is such that $\frac{\nu_1}{\nu_2} = \tilde r^{1/p_1‐1/p_2}$. Since $m^{2/q}k^{2/\sigma}\le n\le mk^{2/\sigma}$, we have $1\le r\le m$. By Lemma \ref{incl_1}, $\nu_1^{1‐\tilde \lambda} \nu_2^{\tilde \lambda} r^{‐1/2}V^{m,k}_{r,1} \subset 4(\nu_1B_{p_1,\theta_1}^{m,k} \cap \nu_2B_{p_2,\theta_2}^{m,k})$. In addition, $n\ge m^{\frac 2q}k^{\frac{2}{\sigma}} r^{1‐\frac 2q}$. Hence
$$
d_n(\nu_1B_{p_1,\theta_1}^{m,k} \cap \nu_2B_{p_2,\theta_2}^{m,k}, \, l_{q,\sigma}^{m,k}) \gtrsim d_n(\nu_1^{1‐\tilde \lambda} \nu_2^{\tilde \lambda} r^{‐\frac 12}V^{m,k}_{r,1}, \, l_{q,\sigma}^{m,k}) \stackrel{(\ref{dn_vmk1})}{\underset{q,\sigma}{\gtrsim}} $$$$\gtrsim\nu_1^{1‐\tilde \lambda} \nu_2^{\tilde \lambda} r^{‐\frac 12} n^{‐\frac 12}m^{\frac 1q} k^{\frac{1}{\sigma}} r^{\frac 12} = \nu_1^{1‐\tilde \lambda} \nu_2^{\tilde \lambda} n^{‐\frac 12}m^{\frac 1q} k^{\frac{1}{\sigma}}.
$$
This completes the proof.
\end{proof}
\begin{Lem}
\label{lemma10}
Let $q\ge 2$, $\sigma\ge 2$.
\begin{enumerate}
\item Let $mk^{2/\sigma}\le n \le \frac{mk}{2}$, $1\le \frac{\nu_1}{\nu_2} \le m^{\frac{1}{p_1} ‐\frac{1}{p_2}}$ or $m^{\frac{1}{p_1} ‐\frac{1}{p_2}}\le \frac{\nu_1}{\nu_2} \le  1$, $\tilde \lambda\in [0, \, 1]$, $\frac 12 = \frac{1‐\tilde \lambda}{p_1} +\frac{\tilde \lambda}{p_2}$. Then
$$
d_n(\nu_1B_{p_1,\theta_1}^{m,k} \cap \nu_2B_{p_2,\theta_2}^{m,k}, \, l^{m,k}_{q,\sigma}) \underset{q,\sigma}{\gtrsim} \nu_1^{1‐\tilde \lambda}\nu_2^{\tilde \lambda} n^{‐\frac 12}m^{\frac 1q}k^{\frac{1}{\sigma}}.
$$

\item Let $m^{2/q}k\le n \le \frac{mk}{2}$, $1\le \frac{\nu_1}{\nu_2} \le k^{\frac{1}{\theta_1} ‐\frac{1}{\theta_2}}$ or $k^{\frac{1}{\theta_1} ‐\frac{1}{\theta_2}}\le \frac{\nu_1}{\nu_2} \le 1$, $\tilde \mu\in [0, \, 1]$, $\frac 12 = \frac{1‐\tilde \mu}{\theta_1} +\frac{\tilde \mu}{\theta_2}$. Then
$$
d_n(\nu_1B_{p_1,\theta_1}^{m,k} \cap \nu_2B_{p_2,\theta_2}^{m,k}, \, l^{m,k}_{q,\sigma}) \underset{q,\sigma}{\gtrsim} \nu_1^{1‐\tilde \mu}\nu_2^{\tilde \mu} n^{‐\frac 12}m^{\frac 1q}k^{\frac{1}{\sigma}}.
$$
\end{enumerate}
\end{Lem}

\begin{proof}
We prove assertion 1 (assertion 2 is similar).

Let $r = \lfloor \tilde r\rfloor$, $\tilde r = m^\alpha$, where $\alpha \in [0, \, 1]$ is such that $\frac{\nu_1}{\nu_2} = \tilde r^{\frac{1}{p_1} ‐\frac{1}{p_2}}$. By Lemma \ref{incl_1}, $\nu_1^{1‐\tilde \lambda}\nu_2^{\tilde \lambda} r^{‐\frac 12} V_{r,1}^{m,k} \subset 4(\nu_1B_{p_1,\theta_1}^{m,k} \cap \nu_2B_{p_2,\theta_2}^{m,k})$. Since $n\ge mk^{2/\sigma}$, we have $n\ge m^{\frac 2q} k^{\frac{2}{\sigma}} r^{1‐\frac 2q}$. Therefore,
$$
d_n(\nu_1B_{p_1,\theta_1}^{m,k} \cap \nu_2B_{p_2,\theta_2}^{m,k}, \, l^{m,k}_{q,\sigma}) \gtrsim d_n(\nu_1^{1‐\tilde \lambda}\nu_2^{\tilde \lambda} r^{‐\frac 12} V_{r,1}^{m,k}, \, l^{m,k}_{q,\sigma}) \stackrel{(\ref{dn_vmk1})}{\underset{q,\sigma}{\gtrsim}}
$$
$$
\gtrsim \nu_1^{1‐\tilde \lambda}\nu_2^{\tilde \lambda} r^{‐\frac 12} n^{‐\frac 12} m^{\frac 1q} k^{\frac{1}{\sigma}} r^{\frac 12} = \nu_1^{1‐\tilde \lambda}\nu_2^{\tilde \lambda}n^{‐\frac 12} m^{\frac 1q} k^{\frac{1}{\sigma}}.
$$
This completes the proof.
\end{proof}

\section{Estimates for the widths of an intersection of finite‐dimensional balls}

We prove Theorem \ref{main}. The upper estimate follows from Lemma \ref{emb}. Let us prove the lower estimate.

If $n\le m^{\frac 2q}k^{\frac{2}{\sigma}}$, we use Lemma \ref{lemma01} and (\ref{dn_1})--(\ref{dn_3}).

In what follows, $m^{\frac 2q}k^{\frac{2}{\sigma}} \le n \le \frac{mk}{2}$. Here we consider all cases up to rearrangement of indices  1 and 2.

\vskip 0.2cm

\textbf{1. Case $p_1$, $p_2$, $\theta_1$, $\theta_2 \in [1, \, 2]$.} From Lemma \ref{lemma01} and (\ref{dn_1}) it follows that $$d_n(\nu_1 B_{p_1,\theta_1}^{m,k}\cap \nu_2 B_{p_2,\theta_2}^{m,k}, \, l_{q,\sigma} ^{m,k}) \underset{q,\sigma}{\gtrsim} \min _{j=1,2} \Phi_j(m, \, k, \, n).$$

\textbf{2. Case $p_1$, $p_2 \in [2, \, q]$, $\theta_1$, $\theta_2 \in [2, \, \sigma ]$; we suppose that one of the following conditions holds: a) $q>2$, $\lambda_{p_i,q}\le \lambda_{\theta_i, \sigma}$, $i=1, \, 2$, b) $\sigma>2$, $\lambda_{p_i,q}\ge \lambda_{\theta_i, \sigma}$, $i=1, \, 2$.}

We prove that $d_n(\nu_1 B_{p_1,\theta_1}^{m,k}\cap \nu_2 B_{p_2,\theta_2}^{m,k}, \, l_{q,\sigma} ^{m,k}) \underset{q,\sigma}{\gtrsim} \min _{j=1,2} \Phi_j(m, \, k, \, n)$.

Consider case a); case b) is similar.

For $m^{\frac 2q}k^{\frac{2}{\sigma}}\le n \le mk^{\frac{2}{\sigma}}$ we use Lemma \ref{lemma02} and the estimate
$$
\Phi_i(m, \, k, \, n)=\nu_i d_n(B_{p_i,\theta_i}^{m,k}, \, l_{q,\sigma} ^{m,k}) \stackrel{(\ref{dn_2})}{\underset{q,\sigma}{\asymp}} \nu_i(n^{‐\frac 12} m^{\frac 1q} k^{\frac{1}{\sigma}}) ^{\frac{1/p_i‐1/q}{1/2‐1/q}}, \quad i=1,\, 2.
$$

For $n\ge mk^{\frac{2}{\sigma}}$ we use Lemma \ref{lemma03} and the estimate
$$
\Phi_i(m, \, k, \, n)=\nu_id_n(B_{p_i,\theta_i}^{m,k}, \, l_{q,\sigma} ^{m,k}) \stackrel{(\ref{dn_2})}{\underset{q,\sigma}{\asymp}} \nu_i m^{\frac 1q‐\frac{1}{p_i}}(n^{‐\frac 12} m^{\frac 12} k^{\frac{1}{\sigma}}) ^{\frac{1/\theta_i‐1/\sigma}{1/2‐1/\sigma}}, \quad i=1,\, 2.
$$

\textbf{3. Case $p_1$, $p_2 \in [2, \, q]$, $\theta_1$, $\theta_2 \in [2, \, \sigma]$, $\lambda_{p_1,q}< \lambda_{\theta_1, \sigma}$, $\lambda_{p_2,q}> \lambda_{\theta_2, \sigma}$.}

From the last two inequalities it follows that
\begin{align}
\label{p1p2t1t2}
\frac{1/p_1‐1/p_2}{1/2‐1/q}<\frac{1/\theta_1‐1/\theta_2}{1/2‐1/\sigma}.
\end{align}
We claim that $$d_n(\nu_1B^{m,k}_{p_1,\theta_1} \cap \nu_2 B_{p_2,\theta_2}^{m,k}, \, l^{m,k}_{q,\sigma}) \underset{q,\sigma}{\gtrsim} \min _{j=1,2,5} \Phi_j(m, \, k, \, n).$$

{\bf Subcase $m^{\frac 2q} k^{\frac{2}{\sigma}} \le n\le \min \{m k^{\frac{2}{\sigma}}, \, m^{\frac 2q} k\}$}.  We have
$$
\Phi_1(m, \, k, \, n)=\nu_1d_n(B^{m,k}_{p_1,\theta_1}, \, l^{m,k}_{q,\sigma}) \stackrel{(\ref{dn_2})}{\underset{q,\sigma}{\asymp}} \nu_1 (n^{‐\frac 12} m^{\frac 1q} k^{\frac{1}{\sigma}})^{\frac{1/p_1‐1/q}{1/2‐1/q}},
$$
$$
\Phi_2(m, \, k, \, n)=\nu_2d_n(B^{m,k}_{p_2,\theta_2}, \, l^{m,k}_{q,\sigma}) \stackrel{(\ref{dn_3})}{\underset{q,\sigma}{\asymp}} \nu_2 (n^{‐\frac 12} m^{\frac 1q} k^{\frac{1}{\sigma}})^{\frac{1/\theta_2‐1/\sigma}{1/2‐1/\sigma}},
$$
$$
\Phi_5(m, \, k, \, n)=\nu_1^{1‐\lambda} \nu_2^\lambda d_n(B^{m,k}_{p,\theta}, \, l^{m,k}_{q,\sigma}) \stackrel{(\ref{dn_2})}{\underset{q,\sigma}{\asymp}}$$$$\asymp \nu_1^{1‐\lambda}\nu_2^\lambda (n^{‐\frac 12} m^{\frac 1q} k^{\frac{1}{\sigma}})^{\frac{1/p‐1/q}{1/2‐1/q}} \stackrel{(\ref{1234})}{=} \nu_1^{1‐\lambda}\nu_2^\lambda (n^{‐\frac 12} m^{\frac 1q} k^{\frac{1}{\sigma}})^{\frac{1/\theta‐1/\sigma}{1/2‐1/\sigma}}
$$
(the last equatity follows from the choice of $\lambda$ in the definition of $\Phi_5(m, \, k, \, n)$; see the statement of Theorem \ref{main}).
Also from the definition of $\lambda$ it follows that
\begin{align}
\label{1le_int} \nu_1 (n^{‐\frac 12} m^{\frac 1q} k^{\frac{1}{\sigma}})^{\frac{1/p_1‐1/q}{1/2‐1/q}} \le \nu_1^{1‐\lambda}\nu_2^\lambda (n^{‐\frac 12} m^{\frac 1q} k^{\frac{1}{\sigma}})^{\frac{1/p‐1/q}{1/2‐1/q}} \; \Leftrightarrow \; \frac{\nu_1}{\nu_2} \le (n^{\frac 12} m^{‐\frac 1q} k^{‐\frac{1}{\sigma}})^{\frac{1/p_1‐1/p_2} {1/2‐1/q}},
\end{align}
\begin{align}
\label{int_le2} \nu_1^{1‐\lambda}\nu_2^\lambda (n^{‐\frac 12} m^{\frac 1q} k^{\frac{1}{\sigma}})^{\frac{1/\theta‐1/\sigma}{1/2‐1/\sigma}} \le \nu_2 (n^{‐\frac 12} m^{\frac 1q} k^{\frac{1}{\sigma}})^{\frac{1/\theta_2‐1/\sigma}{1/2‐1/\sigma}} \; \Leftrightarrow \; \frac{\nu_1}{\nu_2} \le (n^{\frac 12} m^{‐\frac 1q} k^{‐\frac{1}{\sigma}})^{\frac{1/\theta_1‐1/\theta_2}{1/2‐1/\sigma}}.
\end{align}
From (\ref{p1p2t1t2}) and the condition $n^{\frac 12} m^{‐\frac 1q} k^{‐\frac{1}{\sigma}}\ge 1$ it follows that $$(n^{\frac 12} m^{‐\frac 1q} k^{‐\frac{1}{\sigma}})^{\frac{1/p_1‐1/p_2} {1/2‐1/q}} \le (n^{\frac 12} m^{‐\frac 1q} k^{‐\frac{1}{\sigma}})^{\frac{1/\theta_1‐1/\theta_2}{1/2‐1/\sigma}}.$$

By (\ref{1le_int}), (\ref{int_le2}), for $\frac{\nu_1}{\nu_2} \le (n^{\frac 12} m^{‐\frac 1q} k^{‐\frac{1}{\sigma}})^{\frac{1/p_1‐1/p_2} {1/2‐1/q}}$ we have
$$
\min _{j=1,2,5} \Phi_j(m, \, k, \, n) \underset{q,\sigma}{\asymp} \nu_1 (n^{‐\frac 12} m^{\frac 1q} k^{\frac{1}{\sigma}})^{\frac{1/p_1‐1/q}{1/2‐1/q}}.
$$
Hence, in order to estimare $d_n(\nu_1B^{m,k}_{p_1,\theta_1} \cap \nu_2 B_{p_2,\theta_2}^{m,k}, \, l^{m,k}_{q,\sigma})$ from below, it suffices to apply Lemma \ref{lemma02}.
The case $\frac{\nu_1}{\nu_2} \ge (n^{\frac 12} m^{‐\frac 1q} k^{‐\frac{1}{\sigma}})^{\frac{1/\theta_1‐1/\theta_2}{1/2‐1/\sigma}}$ can be considered similarly. Let $$(n^{\frac 12} m^{‐\frac 1q} k^{‐\frac{1}{\sigma}})^{\frac{1/p_1‐1/p_2} {1/2‐1/q}} \le \frac{\nu_1}{\nu_2} \le (n^{\frac 12} m^{‐\frac 1q} k^{‐\frac{1}{\sigma}})^{\frac{1/\theta_1‐1/\theta_2}{1/2‐1/\sigma}}.$$ 
From (\ref{1le_int}), (\ref{int_le2}) it follows that $$\min _{j=1,2,5} \Phi_j(m, \, k, \, n) \underset{q,\sigma}{\asymp} \nu_1^{1‐\lambda}\nu_2^\lambda (n^{‐\frac 12} m^{\frac 1q} k^{\frac{1}{\sigma}})^{\frac{1/p‐1/q}{1/2‐1/q}}.$$
It suffices to apply Lemma \ref{lemma04}.

{\bf Subcase $km^{\frac 2q}\le n \le mk^{\frac{2}{\sigma}}$.} We prove that
$$
(n^{\frac 12} m^{‐\frac 1q} k^{‐\frac{1}{\sigma}})^{\frac{1/p_1‐1/p_2}{1/2‐1/q}} \le k^{\frac{1}{\theta_1}‐\frac{1}{\theta_2}} (n^{\frac 12} m^{‐\frac 1q} k^{‐\frac{1}{2}}) ^{\frac{1/p_1‐1/p_2}{1/2‐1/q}}.
$$
It is equivalent to $1\le k^{\frac{1}{\theta_1} ‐\frac{1}{\theta_2} ‐ \frac{1/2‐1/\sigma}{1/2‐1/q}\left(\frac{1}{p_1}‐\frac{1}{p_2}\right)}$; this follows from (\ref{p1p2t1t2}).

We apply (\ref{dn_2}), (\ref{dn_3}) and compare the right‐hand sides of the corresponding order equalities, taking into account (\ref{1234}) (as in the previous case).

If $\frac{\nu_1}{\nu_2} \le  (n^{‐\frac 12} m^{\frac 1q} k^{\frac{1}{\sigma}}) ^{\frac{1/p_2‐1/p_1}{1/2‐1/q}}$,
we have
$$
\min_{j=1,2,5} \Phi_j(m,\, k, \, n) \underset{q,\sigma}{\asymp} \nu_1 (n^{‐\frac 12} m^{\frac 1q} k^{\frac{1}{\sigma}})^{\frac{1/p_1‐1/q}{1/2‐1/q}};
$$
it remains to apply Lemma \ref{lemma02}.
If $\frac{\nu_1}{\nu_2} \ge k^{\frac{1}{\theta_1}‐\frac{1}{\theta_2}} (n^{\frac 12} m^{‐\frac 1q} k^{‐\frac{1}{2}}) ^{\frac{1/p_1‐1/p_2}{1/2‐1/q}}$, then 
$$
\min_{j=1,2,5} \Phi_j(m,\, k, \, n) \underset{q,\sigma}{\asymp} \nu_2 k^{\frac{1}{\sigma}‐ \frac{1}{\theta_2}}(n^{‐\frac 12} m^{\frac 1q} k^{\frac{1}{2}})^{\frac{1/p_2‐1/q}{1/2‐1/q}};
$$
we apply Lemma \ref{lemma03} (rearranging the indices  1 and 2). If
$$
(n^{\frac 12} m^{‐\frac 1q} k^{‐\frac{1}{\sigma}})^{\frac{1/p_1‐1/p_2}{1/2‐1/q}} \le \frac{\nu_1}{\nu_2}\le k^{\frac{1}{\theta_1}‐\frac{1}{\theta_2}} (n^{\frac 12} m^{‐\frac 1q} k^{‐\frac{1}{2}}) ^{\frac{1/p_1‐1/p_2}{1/2‐1/q}},
$$
then
$$
\min_{j=1,2,5} \Phi_j(m,\, k, \, n) \underset{q,\sigma}{\asymp} \nu_1^{1‐\lambda}\nu_2^\lambda (n^{‐\frac 12} m^{\frac 1q} k^{\frac{1}{\sigma}})^{\frac{1/p‐1/q}{1/2‐1/q}};
$$
we apply Lemma \ref{lemma05}.

{\bf Subcase $mk^{\frac{2}{\sigma}} \le n \le km^{\frac 2q}$} is similar.

{\bf Subcase $\max \{mk^{\frac{2}{\sigma}}, \, km^{\frac 2q}\} \le n \le \frac{mk}{2}$}. 

Notice that 
$$
m^{\frac{1}{p_1}‐\frac{1}{p_2}} (n^{\frac 12} m^{‐\frac 12} k^{‐\frac{1}{\sigma}})^{\frac{1/\theta_1‐1/\theta_2}{1/2‐1/\sigma}} \le k^{\frac{1}{\theta_1} ‐\frac{1}{\theta_2}}(n^{\frac 12} m^{‐\frac 1q} k^{‐\frac{1}{2}}) ^{\frac{1/p_1‐1/p_2}{1/2‐1/q}}.
$$
Indeed, it is equivalent to
$$
(mkn^{‐1})^{\frac{1/\theta_1‐1/\theta_2}{1/2‐1/\sigma}} \ge (mkn^{‐1})^{\frac{1/p_1‐1/p_2}{1/2‐1/q}}.
$$
This inequality follows from (\ref{p1p2t1t2}).

We apply (\ref{dn_2}), (\ref{dn_3}), compare the right‐hand sides of these order equalities and take into account (\ref{1234}). The following assertions hold:
\begin{itemize}
\item if $\frac{\nu_1}{\nu_2} \le m^{\frac{1}{p_1}‐\frac{1}{p_2}} (n^{\frac 12} m^{‐\frac 12} k^{‐\frac{1}{\sigma}})^{\frac{1/\theta_1‐1/\theta_2}{1/2‐1/\sigma}}$, then $$\min_{j=1,2,5}\Phi_j(m, \, n, \, k) \underset{q,\sigma}{\asymp}\nu_1 m^{\frac{1}{q}‐\frac{1}{p_1}}(n^{‐\frac 12} m^{\frac 12} k^{\frac{1}{\sigma}})^{\frac{1/\theta_1‐1/\sigma}{1/2‐1/\sigma}};$$ we use Lemma \ref{lemma03};
\item if $\frac{\nu_1}{\nu_2} \ge k^{\frac{1}{\theta_1} ‐\frac{1}{\theta_2}}(n^{\frac 12} m^{‐\frac 1q} k^{‐\frac{1}{2}}) ^{\frac{1/p_1‐1/p_2}{1/2‐1/q}}$, then $$\min_{j=1,2,5}\Phi_j(m, \, n, \, k) \underset{q,\sigma}{\asymp}\nu_2 k^{\frac{1}{\sigma}‐ \frac{1}{\theta_2}}(n^{‐\frac 12} m^{\frac 1q} k^{\frac{1}{2}})^{\frac{1/p_2‐1/q}{1/2‐1/q}};$$ we use Lemma \ref{lemma03};
\item if $m^{\frac{1}{p_1}‐\frac{1}{p_2}} (n^{\frac 12} m^{‐\frac 12} k^{‐\frac{1}{\sigma}})^{\frac{1/\theta_1‐1/\theta_2}{1/2‐1/\sigma}} \le \frac{\nu_1}{\nu_2} \le k^{\frac{1}{\theta_1} ‐\frac{1}{\theta_2}}(n^{\frac 12} m^{‐\frac 1q} k^{‐\frac{1}{2}}) ^{\frac{1/p_1‐1/p_2}{1/2‐1/q}}$, then
$$
\min_{j=1,2,5}\Phi_j(m, \, n, \, k) \underset{q,\sigma}{\asymp} \nu_1^{1‐\lambda}\nu_2^\lambda m^{\frac 1q‐\frac 1p}(n^{‐\frac 12} m^{\frac 12} k^{\frac{1}{\sigma}})^{\frac{1/\theta‐1/\sigma}{1/2‐1/\sigma}};
$$
we apply Lemma \ref{lemma06}.
\end{itemize}

\textbf{4a. Case $p_1, \, p_2\in [2, \, q]$, $\theta_1\in (2, \, \sigma]$, $\theta_2\in [1, \, 2)$, $\lambda_{p_1,q} \le \lambda _{\theta_1,\sigma}$.} We claim that
$$
d_n(\nu_1B_{p_1,\theta_1}^{m,k}\cap \nu_2B_{p_2,\theta_2}^{m,k}, \, l_{q,\sigma}^{m,k}) \underset{q,\sigma}{\asymp} \min _{j=1,2,4} \Phi_j(m, \, k, \, n).
$$

{\bf Subcase $m^{2/q}k^{2/\sigma} \le n\le mk^{2/\sigma}$.} We have
$$
d_n(\nu_1B_{p_1,\theta_1}^{m,k}, \, l_{q,\sigma}^{m,k}) \stackrel{(\ref{dn_2})}{\underset{q,\sigma}{\asymp}} \nu_1(n^{‐\frac 12}m^{\frac 1q} k^{\frac{1}{\sigma}})^{\frac{1/p_1‐1/q}{1/2‐1/q}}=:a,
$$
$$
d_n(\nu_2B_{p_2,\theta_2}^{m,k}, \, l_{q,\sigma}^{m,k}) \stackrel{(\ref{dn_2})}{\underset{q,\sigma}{\asymp}} \nu_2(n^{‐\frac 12}m^{\frac 1q} k^{\frac{1}{\sigma}})^{\frac{1/p_2‐1/q}{1/2‐1/q}}=:b,
$$
$$
d_n(\nu_1^{1‐\tilde \mu}\nu_2^{\tilde \mu}B_{\tilde p,2}^{m,k}, \, l_{q,\sigma}^{m,k}) \stackrel{(\ref{dn_2})}{\underset{q,\sigma}{\asymp}} \nu_1^{1‐\tilde \mu} \nu_2^{\tilde \mu}(n^{‐\frac 12}m^{\frac 1q} k^{\frac{1}{\sigma}})^{\frac{1/\tilde p‐1/q}{1/2‐1/q}}= a^{1‐\tilde \mu}b^{\tilde \mu}.
$$
It suffices to prove that for $\frac{\nu_1}{\nu_2} \le (n^{\frac 12} m^{‐\frac 1q}k^{‐\frac{1}{\sigma}})^{\frac{1/p_1‐1/p_2}{1/2‐1/q}}$ 
$$
d_n(\nu_1B_{p_1,\theta_1}^{m,k}\cap \nu_2 B_{p_2,\theta_2}^{m,k}, \, l_{q,\sigma}^{m,k}) \underset{q,\sigma}{\gtrsim} \nu_1(n^{‐\frac 12}m^{\frac 1q} k^{\frac{1}{\sigma}})^{\frac{1/p_1‐1/q}{1/2‐1/q}},
$$
and for $\frac{\nu_1}{\nu_2} \ge (n^{\frac 12} m^{‐\frac 1q}k^{‐\frac{1}{\sigma}})^{\frac{1/p_1‐1/p_2}{1/2‐1/q}}$,
$$
d_n(\nu_1B_{p_1,\theta_1}^{m,k}\cap \nu_2 B_{p_2,\theta_2}^{m,k}, \, l_{q,\sigma}^{m,k}) \underset{q,\sigma}{\gtrsim} \nu_2(n^{‐\frac 12}m^{\frac 1q} k^{\frac{1}{\sigma}})^{\frac{1/p_2‐1/q}{1/2‐1/q}}.
$$
It follows from Lemma \ref{lemma02}.

{\bf Subcase $mk^{2/\sigma} \le n \le \frac{mk}{2}$.}
Notice that $m^{\frac{1}{p_1}‐\frac{1}{p_2}} (n^{\frac 12}m^{‐\frac 12}k^{‐\frac{1}{\sigma}})^{\frac{1/\theta_1‐1/\theta_2}{1/2‐1/\sigma}}\le m^{\frac{1}{p_1}‐\frac{1}{p_2}}$.

We apply (\ref{dn_2}) and compare the right‐hand sides of the corresponding order equalities. If $\frac{\nu_1}{\nu_2} \le m^{\frac{1}{p_1}‐\frac{1}{p_2}} (n^{\frac 12}m^{‐\frac 12}k^{‐\frac{1}{\sigma}})^{\frac{1/\theta_1‐1/\theta_2}{1/2‐1/\sigma}}$, then $$\min _{j=1,2,4}\Phi_j(m, \, k, \, n) \underset{q,\sigma}{\asymp} \nu_1 m^{\frac 1q‐\frac{1}{p_1}} (n^{‐\frac 12}m^{\frac 12} k^{\frac{1}{\sigma}}) ^{\frac{1/\theta_1‐1/\sigma}{1/2‐1/\sigma}};$$ we use Lemma \ref{lemma03}. If $\frac{\nu_1}{\nu_2}\ge m^{\frac{1}{p_1} ‐\frac{1}{p_2}}$, then $$\min _{j=1,2,4}\Phi_j(m, \, k, \, n) \underset{q,\sigma}{\asymp} \nu_2 m^{\frac 1q‐\frac{1}{p_2}+\frac 12} k^{\frac{1}{\sigma}} n^{‐\frac 12};$$ here we use Lemma \ref{lemma07}. If
$$
m^{\frac{1}{p_1}‐\frac{1}{p_2}} (n^{\frac 12}m^{‐\frac 12}k^{‐\frac{1}{\sigma}})^{\frac{1/\theta_1‐1/\theta_2}{1/2‐1/\sigma}} \le \frac{\nu_1}{\nu_2} \le m^{\frac{1}{p_1}‐\frac{1}{p_2}},
$$
we have $$\min _{j=1,2,4}\Phi_j(m, \, k, \, n) \underset{q,\sigma}{\asymp} \nu_1^{1‐\tilde \mu}\nu_2^{\tilde \mu} m^{\frac 1q‐\frac{1}{\tilde p}+\frac 12} k^{\frac{1}{\sigma}} n^{‐\frac 12};$$ now we allpy Lemma \ref{lemma08}.

\textbf{4b. Case $p_1\in (2, \, q]$, $p_2\in [1, \, 2)$, $\theta_1, \, \theta_2\in [2, \, \sigma]$, $\lambda_{p_1,q} \ge \lambda _{\theta_1,\sigma}$} is similar.

\vskip 0.2cm

\textbf{5a. Case $q>2$, $\sigma>2$, $p_1, \, p_2\in [2, \, q]$, $\theta_1\in [2, \, \sigma]$, $\theta_2\in [1, \, 2]$, $\lambda_{p_1,q}>\lambda_{\theta_1,\sigma}$.}

We claim that $$d_n(\nu_1B^{m,k}_{p_1,\theta_1}\cap \nu_2 B^{m,k}_{p_2,\theta_2}, \, l^{m,k}_{q,\sigma}) \underset{q,\sigma}{\gtrsim} \min _{j=1,2,4,5} \Phi_j(m, \, k, \, n).$$

Since $\lambda_{p_1,q}> \lambda_{\theta_1,\sigma}$, $\lambda_{\tilde p, q}\le \lambda_{2,\sigma}$, we have $\Phi_5(m, \, k, \, n)<\infty$ and $\tilde \mu \ge \lambda$. In addition, from $\lambda_{p_1,q}>\lambda_{\theta_1,\sigma}$ and $\theta_2\le 2 \le p_2$ we get
\begin{align}
\label{pp1tt2} \frac{1/p_1‐1/p_2}{1/2‐1/q} > \frac{1/\theta_1‐1/\theta_2}{1/2‐1/\sigma}.
\end{align}

We use (\ref{dn_2}), (\ref{dn_3}) and compare the right‐hand sides of these order equalities taking into account (\ref{1234}) and the inequalities $0\le \lambda \le \tilde \mu\le 1$.

If $m^{\frac 2q}k^{\frac{2}{\sigma}} \le n\le \min \{m^{\frac 2q}k, \, mk^{\frac{2}{\sigma}}\}$, we apply Lemma \ref{lemma02} for $\frac{\nu_1}{\nu_2} \le (n^{\frac 12} m^{‐\frac 1q} k^{‐\frac{1}{\sigma}})^{\frac{1/\theta_1‐1/\theta_2}{1/2‐1/\sigma}}$ or $\frac{\nu_1}{\nu_2} \ge (n^{\frac 12} m^{‐\frac 1q} k^{‐\frac{1}{\sigma}})^{\frac{1/p_1‐1/p_2}{1/2‐1/q}}$, and we use Lemma \ref{lemma04} for $(n^{\frac 12} m^{‐\frac 1q} k^{‐\frac{1}{\sigma}})^{\frac{1/\theta_1‐1/\theta_2}{1/2‐1/\sigma}} \le \frac{\nu_1}{\nu_2} \le (n^{\frac 12} m^{‐\frac 1q} k^{‐\frac{1}{\sigma}})^{\frac{1/p_1‐1/p_2}{1/2‐1/q}}$ (we argue as in Case 3).

{\bf Subcase $mk^{\frac{2}{\sigma}} \le n\le m^{\frac 2q}k$.} 

Notice that
$$
(n^{\frac 12} m^{‐\frac 1q} k^{‐\frac{1}{\sigma}}) ^{\frac{1/\theta_1‐1/\theta_2}{1/2‐1/\sigma}} \stackrel{(\ref{pp1tt2})}{\le} m^{\frac{1}{p_1}‐\frac{1}{p_2}}(n^{\frac 12} m^{‐\frac 12} k^{‐\frac{1}{\sigma}}) ^{\frac{1/\theta_1‐1/\theta_2}{1/2‐1/\sigma}} \le m^{\frac{1}{p_1}‐\frac{1}{p_2}}.
$$

We use (\ref{dn_2}), (\ref{dn_3}) taking into account (\ref{1234}). If $\frac{\nu_1}{\nu_2} \le (n^{\frac 12} m^{‐\frac 1q} k^{‐\frac{1}{\sigma}}) ^{\frac{1/\theta_1‐1/\theta_2}{1/2‐1/\sigma}}$, then $$\min_{j=1,2,4,5} \Phi_j(m,\,k,\,n) \underset{q,\sigma}{\asymp} \nu_1 (n^{‐\frac 12} m^{\frac 1q} k^{\frac{1}{\sigma}}) ^{\frac{1/\theta_1‐1/\sigma}{1/2‐1/\sigma}};$$ we use Lemma \ref{lemma02}. If $\frac{\nu_1}{\nu_2} \ge m^{\frac{1}{p_1} ‐\frac{1}{p_2}}$, then $$\min_{j=1,2,4,5} \Phi_j(m,\,k,\,n) \underset{q,\sigma}{\asymp} \nu_2 n^{‐\frac 12} m^{\frac 1q ‐\frac{1}{p_2} +\frac 12} k^{\frac{1}{\sigma}};$$ we use Lemma \ref{lemma07}. If
$$
(n^{\frac 12} m^{‐\frac 1q} k^{‐\frac{1}{\sigma}}) ^{\frac{1/\theta_1‐1/\theta_2}{1/2‐1/\sigma}} \le \frac{\nu_1}{\nu_2}\le m^{\frac{1}{p_1}‐\frac{1}{p_2}}(n^{\frac 12} m^{‐\frac 12} k^{‐\frac{1}{\sigma}}) ^{\frac{1/\theta_1‐1/\theta_2}{1/2‐1/\sigma}},
$$
then $$\min_{j=1,2,4,5} \Phi_j(m,\,k,\,n) \underset{q,\sigma}{\asymp} \nu_1^{1‐\lambda} \nu_2^\lambda (n^{‐\frac 12} m^{\frac 1q} k^{\frac{1}{\sigma}}) ^{\frac{1/\theta‐1/\sigma}{1/2‐1/\sigma}} = \nu_1^{1‐\lambda} \nu_2^\lambda m^{\frac 1q‐\frac 1p}(n^{‐\frac 12} m^{\frac 12} k^{\frac{1}{\sigma}}) ^{\frac{1/\theta‐1/\sigma}{1/2‐1/\sigma}};$$
we apply Lemma \ref{lemma05}. If
$$
m^{\frac{1}{p_1}‐\frac{1}{p_2}}(n^{\frac 12} m^{‐\frac 12} k^{‐\frac{1}{\sigma}}) ^{\frac{1/\theta_1‐1/\theta_2}{1/2‐1/\sigma}} \le \frac{\nu_1}{\nu_2}\le m^{\frac{1}{p_1}‐\frac{1}{p_2}},
$$
then $$\min_{j=1,2,4,5} \Phi_j(m,\,k,\,n) \underset{q,\sigma}{\asymp} \nu_1^{1‐\tilde \mu} \nu_2^{\tilde\mu} m^{\frac 1q‐\frac{1}{\tilde p}} n^{‐\frac 12} m^{\frac 12} k^{\frac{1}{\sigma}}.$$
we apply Lemma \ref{lemma08}.

{\bf Subcase $m^{2/q}k \le n\le mk^{2/\sigma}$.} 

Notice that
$$k^{\frac{1}{\theta_1}- \frac{1}{\theta_2}} (n^{\frac 12}m^{‐\frac 1q} k^{‐\frac 12}) ^{\frac{1/p_1‐1/p_2}{1/2‐1/q}} \stackrel{(\ref{pp1tt2})}{\le} (n^{\frac 12} m^{‐\frac 1q} k^{‐\frac{1}{\sigma}}) ^{\frac{1/p_1‐1/p_2}{1/2‐1/q}}.$$

If $\frac{\nu_1}{\nu_2} \le k^{\frac{1}{\theta_1} ‐\frac{1}{\theta_2}} (n^{\frac 12}m^{‐\frac 1q} k^{‐\frac 12}) ^{\frac{1/p_1‐1/p_2}{1/2‐1/q}}$, then
$$
\min _{j=1,2,4,5} \Phi_j(m, \, k, \, n) \underset{q,\sigma}{\asymp} \nu_1 k^{\frac{1}{\sigma}‐\frac{1}{\theta_1}}(n^{‐\frac 12} m^{\frac 1q} k^{\frac{1}{2}}) ^{\frac{1/p_1‐1/q}{1/2‐1/q}};
$$
we use Lemma \ref{lemma03}. If $\frac{\nu_1}{\nu_2}\ge (n^{\frac 12} m^{‐\frac 1q} k^{‐\frac{1}{\sigma}}) ^{\frac{1/p_1‐1/p_2}{1/2‐1/q}}$, then
$$
\min _{j=1,2,4,5} \Phi_j(m, \, k, \, n) \underset{q,\sigma}{\asymp} \nu_2 (n^{‐\frac 12} m^{\frac 1q} k^{\frac{1}{\sigma}})^{\frac{1/p_2‐1/q}{1/2‐1/q}};
$$
we apply Lemma \ref{lemma02}. If
$$k^{\frac{1}{\theta_1}‐ \frac{1}{\theta_2}} (n^{\frac 12}m^{‐\frac 1q} k^{‐\frac 12}) ^{\frac{1/p_1‐1/p_2}{1/2‐1/q}} \le \frac{\nu_1}{\nu_2} \le (n^{\frac 12} m^{‐\frac 1q} k^{‐\frac{1}{\sigma}}) ^{\frac{1/p_1‐1/p_2}{1/2‐1/q}},$$
then $$\min _{j=1,2,4,5} \Phi_j(m, \, k, \, n) \underset{q,\sigma}{\asymp} \nu_1^{1‐\lambda} \nu_2^\lambda (n^{‐\frac 12} m^{\frac 1q} k^{\frac{1}{\sigma}}) ^{\frac{1/p‐1/q}{1/2‐1/q}};$$ we apply Lemma \ref{lemma05}.

{\bf Subcase $n\ge \max\{mk^{2/\sigma}, \, m^{2/q}k\}$.} 

Notice that
$$
(n^{\frac 12} m^{‐\frac 1q} k^{‐\frac 12})^{\frac{1/p_1‐1/p_2}{1/2‐1/q}} k^{\frac{1}{\theta_1} ‐\frac{1}{\theta_2}} \stackrel{(\ref{pp1tt2})}{\le} m^{\frac{1}{p_1} ‐\frac{1}{p_2}} (n^{\frac 12} m^{‐\frac 12} k^{‐\frac{1}{\sigma}}) ^{\frac{1/\theta_1‐1/\theta_2}{1/2‐1/\sigma}} \le m^{\frac{1}{p_1} ‐\frac{1}{p_2}}.
$$

If $\frac{\nu_1}{\nu_2} \le (n^{\frac 12} m^{‐\frac 1q} k^{‐\frac 12})^{\frac{1/p_1‐1/p_2}{1/2‐1/q}} k^{\frac{1}{\theta_1} ‐\frac{1}{\theta_2}}$, then 
$$
\min _{j=1,2,4,5} \Phi_j(m, \, k, \, n) \underset{q,\sigma}{\asymp} \nu_1 k^{\frac{1}{\sigma}‐\frac{1}{\theta_1}}(n^{‐\frac 12} m^{\frac 1q} k^{\frac{1}{2}}) ^{\frac{1/p_1‐1/q}{1/2‐1/q}};
$$
we use Lemma \ref{lemma03}.
If $\frac{\nu_1}{\nu_2} \ge m^{\frac{1}{p_1} ‐ \frac{1}{p_2}}$, then 
$$
\min _{j=1,2,4,5} \Phi_j(m, \, k, \, n) \underset{q,\sigma}{\asymp} \nu_2 m^{\frac 1q‐\frac{1}{p_2}} n^{‐\frac 12} m^{\frac 12} k^{\frac{1}{\sigma}};
$$ 
we apply Lemma \ref{lemma07}.
If $m^{\frac{1}{p_1} ‐\frac{1}{p_2}} (n^{\frac 12} m^{‐\frac 12} k^{‐\frac{1}{\sigma}}) ^{\frac{1/\theta_1‐1/\theta_2}{1/2‐1/\sigma}} \le \frac{\nu_1}{\nu_2} \le m^{\frac{1}{p_1} ‐ \frac{1}{p_2}}$, then
$$
\min _{j=1,2,4,5} \Phi_j(m, \, k, \, n) \underset{q,\sigma}{\asymp} \nu_1^{1‐\tilde \mu} \nu_2^{\tilde\mu} n^{‐\frac 12} m^{\frac 1q‐\frac{1}{\tilde p} +\frac 12} k^{\frac{1}{\sigma}};
$$
we use Lemma \ref{lemma08}. If
$$
(n^{\frac 12} m^{‐\frac 1q} k^{‐\frac 12})^{\frac{1/p_1‐1/p_2}{1/2‐1/q}} k^{\frac{1}{\theta_1} ‐\frac{1}{\theta_2}}\le \frac{\nu_1}{\nu_2} \le m^{\frac{1}{p_1} ‐\frac{1}{p_2}} (n^{\frac 12} m^{‐\frac 12} k^{‐\frac{1}{\sigma}}) ^{\frac{1/\theta_1‐1/\theta_2}{1/2‐1/\sigma}},
$$
then 
$$
\min _{j=1,2,4,5} \Phi_j(m, \, k, \, n) \underset{q,\sigma}{\asymp} \nu_1^{1‐\lambda} \nu_2^\lambda m^{\frac 1q‐\frac{1}{p}}(n^{‐\frac 12} m^{\frac 12} k^{\frac{1}{\sigma}}) ^{\frac{1/\theta‐1/\sigma}{1/2‐1/\sigma}};
$$
we apply Lemma \ref{lemma06}.

\vskip 0.2cm

\textbf{5b. Case $q>2$, $\sigma>2$, $p_1\in [2, \, q]$, $p_2\in [1, \, 2]$, $\theta_1, \, \theta_2\in [2, \, \sigma]$, $\lambda_{p_1,q}<\lambda_{\theta_1,\sigma}$} is similar.

\vskip 0.2cm

\textbf{6a. Case $p_1, \, \theta_1, \, \theta_2\in [1, \, 2]$, $p_2\in [2, \, q]$.} We claim that
$$
d_n(\nu_1B^{m,k}_{p_1,\theta_1} \cap \nu_2B^{m,k}_{p_2,\theta_2}, \, l^{m,k}_{q,\sigma}) \underset{q, \sigma}{\gtrsim} \min _{1\le j\le 3} \Phi_j(m,\, k, \, n).
$$

{\bf Subcase $m^{\frac 2q}k^{\frac{2}{\sigma}} \le n\le mk^{2/\sigma}$.} 

Notice that $1\le (n^{\frac 12}m^{‐\frac 1q}k^{‐\frac{1}{\sigma}})^{\frac{1/p_1‐1/p_2}{1/2‐1/q}}$.

If $\frac{\nu_1}{\nu_2}\le 1$, then 
$$
\min _{1\le j\le 3} \Phi(m,\, k, \, n) \underset{q,\sigma}{\asymp} \nu_1 n^{‐\frac 12}m^{\frac 1q}k^{\frac{1}{\sigma}};
$$
we use Lemma \ref{lemma01}. If $\frac{\nu_1}{\nu_2} \ge (n^{\frac 12}m^{‐\frac 1q}k^{‐\frac{1}{\sigma}})^{\frac{1/p_1‐1/p_2}{1/2‐1/q}}$, then
$$
\min _{1\le j\le 3} \Phi(m,\, k, \, n) \underset{q,\sigma}{\asymp} \nu_2 (n^{‐\frac 12}m^{\frac 1q}k^{\frac{1}{\sigma}}) ^{\frac{1/p_2‐1/q}{1/2‐1/q}};
$$
we apply Lemma \ref{lemma02}. If $1\le \frac{\nu_1}{\nu_2} \le (n^{\frac 12}m^{‐\frac 1q}k^{‐\frac{1}{\sigma}})^{\frac{1/p_1‐1/p_2}{1/2‐1/q}}$, then
$$
\min _{1\le j\le 3} \Phi(m,\, k, \, n) \underset{q,\sigma}{\asymp} \nu_1^{1‐\tilde \lambda}\nu_2^{\tilde \lambda} n^{‐\frac 12}m^{\frac 1q}k^{\frac{1}{\sigma}};
$$
we use Lemma \ref{lemma09}.

{\bf Subcase $n\ge mk^{2/\sigma}$.} Notice that $1\le m^{\frac{1}{p_1}‐\frac{1}{p_2}}$.

If $\frac{\nu_1}{\nu_2} \le 1$, then
$$
\min _{1\le j\le 3} \Phi(m,\, k, \, n) \underset{q,\sigma}{\asymp} \nu_1 n^{‐\frac 12}m^{\frac 1q}k^{\frac{1}{\sigma}};
$$ 
we use Lemma \ref{lemma01}. If $\frac{\nu_1}{\nu_2} \ge m^{\frac{1}{p_1} ‐\frac{1}{p_2}}$, then
$$
\min _{1\le j\le 3} \Phi(m,\, k, \, n) \underset{q,\sigma}{\asymp} \nu_2 n^{‐\frac 12}m^{\frac 1q-\frac{1}{p_2}+\frac 12}k^{\frac{1}{\sigma}};
$$ 
we use Lemma \ref{lemma07}. If $1\le \frac{\nu_1}{\nu_2} \le m^{\frac{1}{p_1} ‐\frac{1}{p_2}}$, then
$$
\min _{1\le j\le 3} \Phi(m,\, k, \, n) \underset{q,\sigma}{\asymp} \nu_1^{1‐\tilde \lambda}\nu_2^{\tilde \lambda} n^{‐\frac 12}m^{\frac 1q}k^{\frac{1}{\sigma}};
$$ 
we use Lemma \ref{lemma10}.

\vskip 0.2cm

\textbf{6b. Case $p_1, \, \theta_1, \, p_2\in [1, \, 2]$, $ \theta_2\in [2, \, \sigma]$} is similar.

\vskip 0.2cm

\textbf{7. Case $2<q<\infty$, $2<\sigma < \infty$, $ p_1\in [2,\,  q]$, $\theta_1\in[2, \, \sigma]$, $p_2, \, \theta_2 \in [1, \, 2]$; here one of the following conditions holds: a) $\lambda_{p_1,q}\le\lambda_{\theta_1,\sigma}$, $\tilde \mu > \tilde \lambda$, 
b) $\lambda_{p_1,q}\ge\lambda_{\theta_1,\sigma}$, $\tilde \mu < \tilde \lambda$.}

We consider the case a); the case b) is similar.

First we notice that from $\tilde \mu \ge \tilde \lambda$ it follows that $\tilde \theta\ge 2$ and $\tilde p\le 2$. Indeed, $\frac{1}{\tilde p}‐\frac 12 = (\tilde \mu‐\tilde \lambda)\left(\frac{1}{p_2}‐\frac{1}{p_1}\right)\ge 0$, $\frac 12‐\frac{1}{\tilde \theta} = (\tilde \mu‐\tilde \lambda)\left(\frac{1}{\theta_2}‐\frac{1}{\theta_1}\right)\ge 0$. From $\lambda_{p_1,q}\le \lambda_{\theta_1,\sigma}$ and $\lambda_{2,q}\ge \lambda_{\tilde \theta,\sigma}$ it follows that $\lambda$ is well‐defined and $\lambda\in [0, \, \tilde \lambda]$. In addition, 
\begin{align}
\label{qwe}
\frac{1/p_1‐1/p_2}{1/2‐1/q}\le \frac{1/\theta_1‐1/\theta_2}{1/2‐1/\sigma}. 
\end{align}
Indeed, from $\lambda _{p_1,q}\le \lambda_{\theta_1,\sigma}$ we have $\frac{1/2‐1/p_1}{1/2‐1/q}\ge \frac{1/2‐1/\theta_1}{1/2‐1/\sigma}$, or $\frac{\tilde \lambda (1/p_2‐1/p_1)}{1/2‐1/q} \ge \frac{\tilde \mu (1/\theta_2‐1/\theta_1)}{1/2‐1/\sigma}$. Taking into account that $\tilde \mu> \tilde \lambda \ge 0$, $p_2\le p_1$, $\theta_2\le \theta_1$, we get the desired inequality.

{\bf Subcase $m^{2/q}k^{2/\sigma} \le n \le \min \{mk^{2/\sigma}, \, km^{2/q}\}$.} 

Notice that
$$
(n^{\frac 12}m^{‐\frac 1q}k^{‐\frac{1}{\sigma}})^{\frac{1/p_1‐1/p_2}{1/2‐1/q}} \stackrel{(\ref{qwe})}{\le} (n^{\frac 12}m^{‐\frac 1q}k^{‐\frac{1}{\sigma}})^{\frac{1/\theta_1‐1/\theta_2}{1/2‐1/\sigma}} \le 1.
$$

We apply (\ref{dn_1})‐‐(\ref{dn_3}), compare the right‐hand sides of these order inequalities, and take into account (\ref{1234}) and the inequalities $0\le \lambda \le \tilde \lambda \le \tilde \mu \le 1$.

If $\frac{\nu_1}{\nu_2} \le (n^{\frac 12}m^{‐\frac 1q}k^{‐\frac{1}{\sigma}})^{\frac{1/p_1‐1/p_2}{1/2‐1/q}}$, then 
$$
\min _{1\le j\le 5} \Phi_j(m,\, k, \, n) \underset{q,\sigma}{\asymp} \nu_1(n^{‐\frac 12}m^{\frac 1q}k^{\frac{1}{\sigma}}) ^{\frac{1/p_1‐1/q}{1/2‐1/q}};
$$
we use Lemma \ref{lemma02}.
If $\frac{\nu_1}{\nu_2} \ge 1$, then
$$
\min _{1\le j\le 5} \Phi_j(m,\, k, \, n) \underset{q,\sigma}{\asymp} \nu_2 n^{‐\frac 12}m^{\frac 1q}k^{\frac{1}{\sigma}};
$$
we use Lemma \ref{lemma01}.
If
$$
(n^{\frac 12}m^{‐\frac 1q}k^{‐\frac{1}{\sigma}})^{\frac{1/\theta_1‐1/\theta_2}{1/2‐1/\sigma}} \le \frac{\nu_1}{\nu_2} \le 1,
$$
then
$$
\min _{1\le j\le 5} \Phi_j(m,\, k, \, n) \underset{q,\sigma}{\asymp} \nu_1^{1‐\tilde\mu}\nu_2^{\tilde\mu} n^{‐\frac 12}m^{\frac 1q}k^{\frac{1}{\sigma}};
$$
we apply Lemma \ref{lemma09}. If
$$
(n^{\frac 12}m^{‐\frac 1q}k^{‐\frac{1}{\sigma}})^{\frac{1/p_1‐1/p_2}{1/2‐1/q}} \le \frac{\nu_1}{\nu_2} \le (n^{\frac 12}m^{‐\frac 1q}k^{‐\frac{1}{\sigma}})^{\frac{1/\theta_1‐1/\theta_2}{1/2‐1/\sigma}},
$$
then
$$
\min _{1\le j\le 5} \Phi_j(m,\, k, \, n) \underset{q,\sigma}{\asymp} \nu_1^{1‐\lambda}\nu_2^{\lambda} (n^{‐\frac 12}m^{\frac 1q}k^{\frac{1}{\sigma}})^{\frac{1/p‐1/q}{1/2‐1/q}} = \nu_1^{1‐\lambda}\nu_2^{\lambda} (n^{‐\frac 12}m^{\frac 1q}k^{\frac{1}{\sigma}})^{\frac{1/\theta‐1/\sigma}{1/2‐1/\sigma}};
$$
we use Lemma \ref{lemma04}.

{\bf Subcase $mk^{2/\sigma}\le n\le m^{2/q}k$.} Notice that
$$
m^{\frac{1}{p_1}‐\frac{1}{p_2}} (n^{\frac 12}m^{‐\frac 12} k^{‐\frac{1}{\sigma}}) ^{\frac{1/\theta_1‐1/\theta_2}{1/2‐1/\sigma}} \stackrel{(\ref{qwe})}{\le} (n^{\frac 12} m^{‐\frac 1q} k^{‐\frac{1}{\sigma}}) ^{\frac{1/\theta_1‐1/\theta_2}{1/2‐1/\sigma}} \le 1.
$$

If $\frac{\nu_1}{\nu_2}\ge 1$, then 
$$
\min _{1\le j\le 5} \Phi_j(m,\, k, \, n) \underset{q,\sigma}{\asymp} \nu_2 n^{‐\frac 12} m^{\frac 1q} k^{\frac{1}{\sigma}};
$$
we use Lemma \ref{lemma01}. If $$\frac{\nu_1}{\nu_2} \le m^{\frac{1}{p_1}‐\frac{1}{p_2}} (n^{\frac 12}m^{‐\frac 12} k^{‐\frac{1}{\sigma}}) ^{\frac{1/\theta_1‐1/\theta_2}{1/2‐1/\sigma}},$$ then 
$$
\min _{1\le j\le 5} \Phi_j(m,\, k, \, n) \underset{q,\sigma}{\asymp} \nu_1m^{\frac 1q‐\frac{1}{p_1}} (n^{‐\frac 12}m^{\frac 12} k^{\frac{1}{\sigma}}) ^{\frac{1/\theta_1‐1/\sigma}{1/2‐1/\sigma}};
$$ 
we apply Lemma \ref{lemma03}. If $(n^{\frac 12} m^{‐\frac 1q} k^{‐\frac{1}{\sigma}}) ^{\frac{1/\theta_1‐1/\theta_2}{1/2‐1/\sigma}}\le \frac{\nu_1}{\nu_2} \le 1$, then
$$
\min _{1\le j\le 5} \Phi_j(m,\, k, \, n) \underset{q,\sigma}{\asymp} \nu_1^{1‐\tilde\mu}\nu_2^{\tilde\mu} n^{‐\frac 12} m^{\frac 1q} k^{\frac{1}{\sigma}};
$$ 
we use Lemma \ref{lemma09}.
If
$$
m^{\frac{1}{p_1}‐\frac{1}{p_2}} (n^{\frac 12}m^{‐\frac 12} k^{‐\frac{1}{\sigma}}) ^{\frac{1/\theta_1‐1/\theta_2}{1/2‐1/\sigma}} \le \frac{\nu_1}{\nu_2} \le (n^{\frac 12} m^{‐\frac 1q} k^{‐\frac{1}{\sigma}}) ^{\frac{1/\theta_1‐1/\theta_2}{1/2‐1/\sigma}},
$$
then
$$
\min _{1\le j\le 5} \Phi_j(m,\, k, \, n) \underset{q,\sigma}{\asymp} \nu_1^{1‐\lambda}\nu_2^{\lambda} m^{\frac 1q‐\frac 1p} (n^{‐\frac 12} m^{\frac 12} k^{\frac{1}{\sigma}}) ^{\frac{1/\theta‐1/\sigma}{1/2‐1/\sigma}};
$$
we apply Lemma \ref{lemma05}.

{\bf Subcase $km^{2/q}\le n\le mk^{2/\sigma}$.} 
Notice that
$$
(n^{\frac 12} m^{‐\frac 1q} k^{‐\frac{1}{\sigma}}) ^{\frac{1/p_1‐1/p_2}{1/2‐1/q}} \stackrel{(\ref{qwe})}{\le} (n^{\frac 12} m^{‐\frac 1q} k^{‐\frac{1}{2}})^{\frac{1/p_1‐1/p_2}{1/2‐1/q}} k^{\frac{1}{\theta_1} ‐\frac{1}{\theta_2}} \le k^{\frac{1}{\theta_1} ‐\frac{1}{\theta_2}} \le 1.
$$

If $\frac{\nu_1}{\nu_2} \le (n^{\frac 12} m^{‐\frac 1q} k^{‐\frac{1}{\sigma}}) ^{\frac{1/p_1‐1/p_2}{1/2‐1/q}}$, then
$$
\min _{1\le j\le 5} \Phi_j(m,\, k, \, n) \underset{q,\sigma}{\asymp} \nu_1(n^{‐\frac 12} m^{\frac 1q} k^{\frac{1}{\sigma}}) ^{\frac{1/p_1‐1/q}{1/2‐1/q}};
$$
we apply Lemma \ref{lemma02}. If $\frac{\nu_1}{\nu_2} \ge 1$, then
$$
\min _{1\le j\le 5} \Phi_j(m,\, k, \, n) \underset{q,\sigma}{\asymp} \nu_2n^{‐\frac 12} m^{\frac 1q} k^{\frac{1}{\sigma}};
$$
we use Lemma \ref{lemma01}.
If $k^{\frac{1}{\theta_1} ‐\frac{1}{\theta_2}} \le \frac{\nu_1}{\nu_2}\le 1$, then
$$
\min _{1\le j\le 5} \Phi_j(m,\, k, \, n) \underset{q,\sigma}{\asymp} \nu_1^{1‐\tilde\mu}\nu_2^{\tilde\mu} n^{‐\frac 12} m^{\frac 1q} k^{\frac{1}{\sigma}};
$$ 
we use Lemma \ref{lemma10}. If
$$
(n^{\frac 12} m^{‐\frac 1q} k^{‐\frac{1}{2}})^{\frac{1/p_1‐1/p_2}{1/2‐1/q}} k^{\frac{1}{\theta_1} ‐\frac{1}{\theta_2}} \le \frac{\nu_1}{\nu_2} \le k^{\frac{1}{\theta_1} ‐\frac{1}{\theta_2}},
$$
then
$$
\min _{1\le j\le 5} \Phi_j(m,\, k, \, n) \underset{q,\sigma}{\asymp}\nu_1^{1‐\tilde\lambda}\nu_2^{\tilde\lambda} k^{\frac{1}{\sigma} ‐\frac{1}{\tilde\theta}}n^{‐\frac 12} m^{\frac 1q} k^{\frac{1}{2}};
$$
we use Lemma \ref{lemma08}. If $$(n^{\frac 12} m^{‐\frac 1q} k^{‐\frac{1}{\sigma}}) ^{\frac{1/p_1‐1/p_2}{1/2‐1/q}} \le \frac{\nu_1}{\nu_2}\le (n^{\frac 12} m^{‐\frac 1q} k^{‐\frac{1}{2}})^{\frac{1/p_1‐1/p_2}{1/2‐1/q}} k^{\frac{1}{\theta_1} ‐\frac{1}{\theta_2}},$$ then
$$
\min _{1\le j\le 5} \Phi_j(m,\, k, \, n) \underset{q,\sigma}{\asymp} \nu_1^{1‐\lambda}\nu_2^\lambda(n^{‐\frac 12} m^{\frac 1q} k^{\frac{1}{\sigma}}) ^{\frac{1/p‐1/q}{1/2‐1/q}};
$$
we apply Lemma \ref{lemma05}.

{\bf Subcase $n\ge \max\{mk^{2/\sigma}, \, m^{2/q}k\}$.} 
Notice that
$$
m^{\frac{1}{p_1}‐\frac{1}{p_2}}(n^{\frac 12} m^{‐\frac 12} k^{‐\frac{1}{\sigma}}) ^{\frac{1/\theta_1‐1/\theta_2}{1/2‐1/\sigma}} \stackrel{(\ref{qwe})}{\le} (n^{\frac 12} m^{‐\frac 1q} k^{‐\frac{1}{2}}) ^{\frac{1/p_1‐1/p_2}{1/2‐1/q}} k^{\frac{1}{\theta_1} ‐\frac{1}{\theta_2}} \le k^{\frac{1}{\theta_1} ‐\frac{1}{\theta_2}} \le 1.
$$

If $\frac{\nu_1}{\nu_2} \ge 1$, we have
$$
\min _{1\le j\le 5} \Phi_j(m,\, k, \, n) \underset{q,\sigma}{\asymp} \nu_2 n^{‐\frac 12} m^{\frac 1q} k^{\frac{1}{\sigma}};
$$
 we apply Lemma \ref{lemma01}. If $\frac{\nu_1}{\nu_2}\le m^{\frac{1}{p_1}‐\frac{1}{p_2}}(n^{\frac 12} m^{‐\frac 12} k^{‐\frac{1}{\sigma}}) ^{\frac{1/\theta_1‐1/\theta_2}{1/2‐1/\sigma}}$, then
 $$
 \min _{1\le j\le 5} \Phi_j(m,\, k, \, n) \underset{q,\sigma}{\asymp} \nu_1m^{\frac 1q‐\frac{1}{p_1}} (n^{‐\frac 12}m^{\frac 12} k^{\frac{1}{\sigma}}) ^{\frac{1/\theta_1‐1/\sigma}{1/2‐1/\sigma}};
 $$
we use Lemma \ref{lemma03}. If $k^{\frac{1}{\theta_1} ‐\frac{1}{\theta_2}} \le \frac{\nu_1}{\nu_2} \le 1$, then 
 $$
 \min _{1\le j\le 5} \Phi_j(m,\, k, \, n) \underset{q,\sigma}{\asymp}  \nu_1^{1‐\tilde\mu}\nu_2^{\tilde\mu} n^{‐\frac 12} m^{\frac 1q} k^{\frac{1}{\sigma}};
 $$
 we use Lemma \ref{lemma10}. If $(n^{\frac 12} m^{‐\frac 1q} k^{‐\frac{1}{2}}) ^{\frac{1/p_1‐1/p_2}{1/2‐1/q}} k^{\frac{1}{\theta_1} ‐\frac{1}{\theta_2}} \le \frac{\nu_1}{\nu_2} \le k^{\frac{1}{\theta_1} ‐\frac{1}{\theta_2}}$, then
 $$
 \min _{1\le j\le 5} \Phi_j(m,\, k, \, n) \underset{q,\sigma}{\asymp} \nu_1^{1‐\tilde\lambda}\nu_2^{\tilde\lambda} k^{\frac{1}{\sigma} ‐\frac{1}{\tilde\theta}}n^{‐\frac 12} m^{\frac 1q} k^{\frac{1}{2}};
 $$
 we use Lemma \ref{lemma08}. If
$$
m^{\frac{1}{p_1}‐\frac{1}{p_2}}(n^{\frac 12} m^{‐\frac 12} k^{‐\frac{1}{\sigma}}) ^{\frac{1/\theta_1‐1/\theta_2}{1/2‐1/\sigma}} \le \frac{\nu_1}{\nu_2} \le(n^{\frac 12} m^{‐\frac 1q} k^{‐\frac{1}{2}}) ^{\frac{1/p_1‐1/p_2}{1/2‐1/q}} k^{\frac{1}{\theta_1} ‐\frac{1}{\theta_2}}, 
$$
then
$$
\min _{1\le j\le 5} \Phi_j(m,\, k, \, n) \underset{q,\sigma}{\asymp} \nu_1^{1‐\lambda}\nu_2^{\lambda} m^{\frac 1q‐\frac 1p} (n^{‐\frac 12} m^{\frac 12} k^{\frac{1}{\sigma}}) ^{\frac{1/\theta‐1/\sigma}{1/2‐1/\sigma}};
$$
we apply Lemma \ref{lemma06}.

\vskip 0.2cm

\textbf{8. Case $q>2$, $\sigma>2$, $p_1\in [2, \,  q]$, $\theta_1\in [2, \,  \sigma]$, $p_2, \, \theta_2 \in [1, \,  2]$; we suppose that one of the following conditions holds: a) $\lambda_{p_1,q}\le \lambda_{\theta_1,\sigma}$, $\tilde \mu \le \tilde \lambda$,
b) $\lambda_{p_1,q}\ge \lambda_{\theta_1,\sigma}$, $\tilde \mu \ge \tilde \lambda$.}

Let condition a) hold (case b) is similar).

Since $\tilde \mu\le \tilde \lambda$, we have $\tilde p\ge 2$, $\tilde \theta\le 2$.

We prove that $d_n(\nu_1B_{p_1,\theta_1}^{m,k} \cap \nu_2B_{p_2,\theta_2}^{m,k}, \, l^{m,k}_{q,\sigma})\underset{q,\sigma}{\asymp} \min _{1\le j\le 4} \Phi_j(m, \, k, \, n)$.

{\bf Subcase $m^{2/q}k^{2/\sigma} \le n \le mk^{2/\sigma}$.} Notice that $(n^{\frac 12} m^{‐\frac 1q} k^{‐\frac{1}{\sigma}}) ^{\frac{1/p_1‐1/p_2}{1/2‐1/q}}\le 1$.

If $\frac{\nu_1}{\nu_2} \le (n^{\frac 12} m^{‐\frac 1q} k^{‐\frac{1}{\sigma}}) ^{\frac{1/p_1‐1/p_2}{1/2‐1/q}}$, then
$$
\min _{1\le j\le 4} \Phi_j(m, \, k, \, n) \underset{q,\sigma}{\asymp} \nu_1(n^{‐\frac 12} m^{\frac 1q} k^{\frac{1}{\sigma}}) ^{\frac{1/p_1‐1/q}{1/2‐1/q}};
$$
we use Lemma \ref{lemma02}.
If $\frac{\nu_1}{\nu_2}\ge 1$, then
$$
\min _{1\le j\le 4} \Phi_j(m, \, k, \, n) \underset{q,\sigma}{\asymp} \nu_2n^{‐\frac 12} m^{\frac 1q} k^{\frac{1}{\sigma}};
$$
we apply Lemma \ref{lemma01}.
If $(n^{\frac 12} m^{‐\frac 1q} k^{‐\frac{1}{\sigma}}) ^{\frac{1/p_1‐1/p_2}{1/2‐1/q}} \le \frac{\nu_1}{\nu_2}\le 1$, then 
$$
\min _{1\le j\le 4} \Phi_j(m, \, k, \, n) \underset{q,\sigma}{\asymp} \nu_1^{1‐\tilde \lambda} \nu_2^{\tilde \lambda}n^{‐\frac 12} m^{\frac 1q} k^{\frac{1}{\sigma}};
$$
we use Lemma \ref{lemma09}.

{\bf Subcase $n\ge mk^{2/\sigma}$.} Notice that
$$
m^{\frac{1}{p_1}‐\frac{1}{p_2}} (n^{\frac 12} m^{‐\frac 12} k^{‐\frac{1}{\sigma}}) ^{\frac{1/\theta_1‐1/\theta_2}{1/2‐1/\sigma}} \le m^{\frac{1}{p_1}‐\frac{1}{p_2}}\le 1.
$$

If $\frac{\nu_1}{\nu_2}\ge 1$, then
$$
\min _{1\le j\le 4} \Phi_j(m, \, k, \, n) \underset{q,\sigma}{\asymp} \nu_2n^{‐\frac 12} m^{\frac 1q} k^{\frac{1}{\sigma}};
$$
we use Lemma \ref{lemma01}. If $\frac{\nu_1}{\nu_2} \le m^{\frac{1}{p_1}‐\frac{1}{p_2}} (n^{\frac 12} m^{‐\frac 12} k^{‐\frac{1}{\sigma}}) ^{\frac{1/\theta_1‐1/\theta_2}{1/2‐1/\sigma}}$, then
$$
\min _{1\le j\le 4} \Phi_j(m, \, k, \, n) \underset{q,\sigma}{\asymp} \nu_1m^{\frac 1q‐\frac{1}{p_1}}(n^{‐\frac 12} m^{\frac 12} k^{\frac{1}{\sigma}}) ^{\frac{1/\theta_1‐1/\sigma}{1/2‐1/\sigma}};
$$ 
we use Lemma \ref{lemma03}. If
$$
m^{\frac{1}{p_1}‐\frac{1}{p_2}} (n^{\frac 12} m^{‐\frac 12} k^{‐\frac{1}{\sigma}}) ^{\frac{1/\theta_1‐1/\theta_2}{1/2‐1/\sigma}} \le \frac{\nu_1}{\nu_2} \le m^{\frac{1}{p_1}‐\frac{1}{p_2}},
$$
then
$$
\min _{1\le j\le 4} \Phi_j(m, \, k, \, n) \underset{q,\sigma}{\asymp} \nu_1^{1‐\tilde \mu} \nu_2^{\tilde \mu} m^{\frac 1q ‐\frac{1}{\tilde p}}n^{‐\frac 12} m^{\frac 12} k^{\frac{1}{\sigma}};
$$
we use Lemma \ref{lemma08}. If $m^{\frac{1}{p_1}‐\frac{1}{p_2}} \le \frac{\nu_1}{\nu_2} \le 1$, then
 $$
 \min _{1\le j\le 4} \Phi_j(m, \, k, \, n) \underset{q,\sigma}{\asymp} \nu_1^{1‐\tilde \lambda} \nu_2^{\tilde \lambda}n^{‐\frac 12} m^{\frac 1q} k^{\frac{1}{\sigma}};
 $$
we use Lemma \ref{lemma10}.

\textbf{9a. Case $p_1, \, p_2\in [2, \, q]$, $\theta_1, \, \theta_2\in [1, \,  2]$.}

We claim that $d_n(\nu_1B_{p_1,\theta_1}^{m,k} \cap \nu_2B_{p_2,\theta_2}^{m,k}, \, l^{m,k}_{q,\sigma})\underset{q,\sigma}{\asymp} \min _{j=1, 2} \Phi_j(m, \, k, \, n)$.

If $m^{\frac 2q}k^{\frac{2}{\sigma}}\le n\le mk^{\frac{2}{\sigma}}$, then
$$
d_n(\nu_iB_{p_i,\theta_i}^{m,k}, \, l_{q,\sigma}^{m,k}) \stackrel{(\ref{dn_2})}{\underset{q,\sigma}{\asymp}} \nu_i(n^{\frac 12}m^{‐\frac 1q} k^{‐\frac{1}{\sigma}})^{\frac{1/p_i‐1/q}{1/2‐1/q}};
$$
$$
\nu_1 (n^{\frac 12}m^{‐\frac 1q} k^{‐\frac{1}{\sigma}})^{\frac{1/p_1‐1/q}{1/2‐1/q}} \le \nu_2 (n^{\frac 12}m^{‐\frac 1q} k^{‐\frac{1}{\sigma}})^{\frac{1/p_2‐1/q}{1/2‐1/q}}\; \Leftrightarrow \; \frac{\nu_1}{\nu_2} \le (n^{\frac 12}m^{‐\frac 1q} k^{‐\frac{1}{\sigma}})^{\frac{1/p_1‐1/p_2}{1/2‐1/q}}.
$$
Now we apply Lemma \ref{lemma02}.

Let $n\ge mk^{\frac{2}{\sigma}}$. Then
$$
d_n(\nu_iB_{p_i,\theta_i}^{m,k}, \, l_{q,\sigma}^{m,k}) \stackrel{(\ref{dn_2})}{\underset{q,\sigma}{\asymp}} \nu_i m^{\frac 1q‐\frac{1}{p_i}} n^{‐\frac 12} m^{\frac 12}k^{‐\frac{1}{\sigma}}.
$$
We have
$$
\nu_1 m^{\frac 1q‐\frac{1}{p_1}} n^{‐\frac 12} m^{\frac 12}k^{‐\frac{1}{\sigma}} \le \nu_2 m^{\frac 1q‐\frac{1}{p_2}} n^{‐\frac 12} m^{\frac 12}k^{‐\frac{1}{\sigma}} \; \Leftrightarrow \; \frac{\nu_1}{\nu_2} \le m^{\frac{1}{p_1}‐\frac{1}{p_2}}.
$$
It remains to apply Lemma \ref{lemma07}.

\vskip 0.2cm

\textbf{9b. Case $p_1, \, p_2\in [1, \, 2]$, $\theta_1, \, \theta_2\in [2, \, \sigma]$} is similar.

\vskip 0.2cm

\textbf{10. Case $q>2$, $\sigma>2$, $p_1\in [2, \, q]$, $\theta_1\in [1, \, 2]$, $p_2\in [1, \, 2]$, $\theta_2 \in [2, \, \sigma]$, $\tilde \lambda \ge \tilde \mu$.}

Since $\tilde \lambda \ge \tilde \mu$, we have $\tilde p\ge 2$, $\tilde \theta\ge 2$. It follows from the equations $\frac 12 ‐\frac{1}{\tilde p}= (\tilde \lambda ‐\tilde \mu)\left(\frac{1}{p_2} ‐\frac{1}{p_1}\right)$, $\frac 12 ‐ \frac{1}{\tilde \theta} = (\tilde \lambda ‐\tilde \mu)\left(\frac{1}{\theta_1} ‐\frac{1}{\theta_2}\right)$. In addition, $\Phi_5(m, \, k, \, n)<\infty$ and $\lambda \in [\tilde \mu, \, \tilde \lambda]$, since $\frac{1/\tilde p‐1/q}{1/2‐1/q} \le 1=\frac{1/2‐1/\sigma}{1/2‐1/\sigma}$, $\frac{1/2‐1/q}{1/2‐1/q}=1\ge \frac{1/\tilde \theta‐1/\sigma}{1/2‐1/\sigma}$.

{\bf Subcase $m^{\frac 2q}k^{\frac{2}{\sigma}} \le n\le \min \{mk^{\frac{2}{\sigma}}, \, m^{\frac 2q}k\}$.} 

Notice that 
$$
(n^{\frac 12} m^{‐\frac 1q} k^{‐\frac{1}{\sigma}})^{\frac{1/p_1‐1/p_2}{1/2‐1/q}} \le (n^{\frac 12} m^{‐\frac 1q} k^{‐\frac{1}{\sigma}})^{\frac{1/\theta_1‐1/\theta_2}{1/2‐1/\sigma}}.
$$

 If $\frac{\nu_1}{\nu_2} \le (n^{\frac 12} m^{‐\frac 1q} k^{‐\frac{1}{\sigma}})^{\frac{1/p_1‐1/p_2}{1/2‐1/q}}$ 
or
$\frac{\nu_1}{\nu_2} \ge (n^{\frac 12} m^{‐\frac 1q} k^{‐\frac{1}{\sigma}})^{\frac{1/\theta_1‐1/\theta_2}{1/2‐1/\sigma}}$, we have, respectively, 
$$
\min _{1\le j\le 5} \Phi_j(m, \, k, \, n) \underset{q,\sigma}{\asymp} \nu_1 (n^{‐\frac 12} m^{\frac 1q} k^{\frac{1}{\sigma}}) ^{\frac{1/p_1‐1/q}{1/2‐1/q}},
$$
$$
\min _{1\le j\le 5} \Phi_j(m, \, k, \, n) \underset{q,\sigma}{\asymp} \nu_2 (n^{‐\frac 12} m^{\frac 1q} k^{\frac{1}{\sigma}}) ^{\frac{1/\theta_2‐1/\sigma}{1/2‐1/\sigma}};
$$
now we use Lemma \ref{lemma02}. If
$$
(n^{\frac 12} m^{‐\frac 1q} k^{‐\frac{1}{\sigma}})^{\frac{1/p_1‐1/p_2}{1/2‐1/q}} \le \frac{\nu_1}{\nu_2} \le (n^{\frac 12} m^{‐\frac 1q} k^{‐\frac{1}{\sigma}})^{\frac{1/\theta_1‐1/\theta_2}{1/2‐1/\sigma}},
$$
then
$$
\min _{1\le j\le 5} \Phi_j(m, \, k, \, n) \underset{q,\sigma}{\asymp} \nu_1^{1‐\lambda}\nu_2^{\lambda}(n^{‐\frac 12} m^{\frac 1q} k^{\frac{1}{\sigma}}) ^{\frac{1/p‐1/q}{1/2‐1/q}};
$$
we use Lemma \ref{lemma04}.

{\bf Subcase $mk^{2/\sigma}\le n \le m^{2/q}k$.} 

Notice that $$m^{\frac{1}{p_1} ‐\frac{1}{p_2}}\le m^{\frac{1}{p_1} ‐\frac{1}{p_2}} (n^{\frac 12} m^{‐\frac 12} k^{‐\frac{1}{\sigma}}) ^{\frac{1/\theta_1‐1/\theta_2}{1/2‐1/\sigma}} \le (n^{\frac 12} m^{‐\frac 1q} k^{‐\frac{1}{\sigma}})^{\frac{1/\theta_1‐1/\theta_2}{1/2‐1/\sigma}}.$$

If $\frac{\nu_1}{\nu_2} \le m^{\frac{1}{p_1} ‐\frac{1}{p_2}}$, then 
$$
\min _{1\le j\le 5} \Phi_j(m, \, k, \, n) \underset{q,\sigma}{\asymp} \nu_1 m^{\frac 1q‐\frac{1}{p_1}}n^{‐\frac 12} m^{\frac 12} k^{\frac{1}{\sigma}};
$$
we use Lemma \ref{lemma07}. If $\frac{\nu_1}{\nu_2} \ge (n^{\frac 12} m^{‐\frac 1q} k^{‐\frac{1}{\sigma}})^{\frac{1/\theta_1‐1/\theta_2}{1/2‐1/\sigma}}$, then
$$
\min _{1\le j\le 5} \Phi_j(m, \, k, \, n) \underset{q,\sigma}{\asymp} \nu_2 (n^{‐\frac 12} m^{\frac 1q} k^{\frac{1}{\sigma}}) ^{\frac{1/\theta_2‐1/\sigma}{1/2‐1/\sigma}};
$$ 
we use Lemma \ref{lemma02}. If $$m^{\frac{1}{p_1} ‐\frac{1}{p_2}}\le \frac{\nu_1}{\nu_2} \le m^{\frac{1}{p_1} ‐\frac{1}{p_2}} (n^{\frac 12} m^{‐\frac 12} k^{‐\frac{1}{\sigma}}) ^{\frac{1/\theta_1‐1/\theta_2}{1/2‐1/\sigma}},$$
then
$$
\min _{1\le j\le 5} \Phi_j(m, \, k, \, n) \underset{q,\sigma}{\asymp} \nu_1^{1‐\tilde \mu}\nu_2^{\tilde \mu}m^{\frac 1q‐\frac{1}{\tilde p}}n^{‐\frac 12} m^{\frac 12} k^{\frac{1}{\sigma}};
$$
we use Lemma \ref{lemma08}. If
$$m^{\frac{1}{p_1} ‐\frac{1}{p_2}} (n^{\frac 12} m^{‐\frac 12} k^{‐\frac{1}{\sigma}}) ^{\frac{1/\theta_1‐1/\theta_2}{1/2‐1/\sigma}} \le \frac{\nu_1}{\nu_2} \le (n^{\frac 12} m^{‐\frac 1q} k^{‐\frac{1}{\sigma}}) ^{\frac{1/\theta_1‐1/\theta_2}{1/2‐1/\sigma}},$$
then
$$
\min _{1\le j\le 5} \Phi_j(m, \, k, \, n) \underset{q,\sigma}{\asymp} \nu_1^{1‐\lambda}\nu_2^{\lambda}m^{\frac 1q‐\frac 1p}(n^{‐\frac 12} m^{\frac 12} k^{\frac{1}{\sigma}}) ^{\frac{1/\theta‐1/\sigma}{1/2‐1/\sigma}};
$$
we apply Lemma \ref{lemma05} and (\ref{1234}).

{\bf Subcase $m^{2/q}k\le n \le mk^{2/\sigma}$} is similar.

{\bf Subcase $\max\{m^{2/q}k, \, mk^{2/\sigma}\} \le n \le \frac{mk}{2}$.} Notice that $$m^{\frac{1}{p_1} ‐\frac{1}{p_2}} \le m^{\frac{1}{p_1} ‐\frac{1}{p_2}} (n^{\frac 12} m^{‐\frac 12} k^{‐\frac{1}{\sigma}}) ^{\frac{1/\theta_1‐1/\theta_2}{1/2‐1/\sigma}} \le (n^{\frac 12} m^{‐\frac 1q} k^{‐\frac{1}{2}}) ^{\frac{1/p_1‐1/p_2}{1/2‐1/q}} k^{\frac{1}{\theta_1}‐\frac{1}{\theta_2}} \le k^{\frac{1}{\theta_1}‐\frac{1}{\theta_2}}.$$

 If $\frac{\nu_1}{\nu_2} \le m^{\frac{1}{p_1} ‐\frac{1}{p_2}}$ or $\frac{\nu_1}{\nu_2} \ge  k^{\frac{1}{\theta_1}‐\frac{1}{\theta_2}}$, we get, respectively,
 $$
 \min _{1\le j\le 5} \Phi_j(m, \, k, \, n) \underset{q,\sigma}{\asymp} \nu_1 m^{\frac 1q‐\frac{1}{p_1}}n^{‐\frac 12} m^{\frac 12} k^{\frac{1}{\sigma}},
 $$
 $$
 \min _{1\le j\le 5} \Phi_j(m, \, k, \, n) \underset{q,\sigma}{\asymp} \nu_2 k^{\frac{1}{\sigma}‐\frac{1}{\theta_2}}n^{‐\frac 12} m^{\frac 1q} k^{\frac{1}{2}};
 $$
then we use Lemma \ref{lemma07}.
If $$m^{\frac{1}{p_1} ‐\frac{1}{p_2}}\le \frac{\nu_1}{\nu_2} \le m^{\frac{1}{p_1} ‐\frac{1}{p_2}} (n^{\frac 12} m^{‐\frac 12} k^{‐\frac{1}{\sigma}}) ^{\frac{1/\theta_1‐1/\theta_2}{1/2‐1/\sigma}}$$ or $$(n^{\frac 12} m^{‐\frac 1q} k^{‐\frac{1}{2}}) ^{\frac{1/p_1‐1/p_2}{1/2‐1/q}} k^{\frac{1}{\theta_1}‐\frac{1}{\theta_2}}\le \frac{\nu_1}{\nu_2}\le k^{\frac{1}{\theta_1}‐\frac{1}{\theta_2}},$$ then
we get, respectively,
$$
\min _{1\le j\le 5} \Phi_j(m, \, k, \, n) \underset{q,\sigma}{\asymp} \nu_1^{1‐\tilde \mu}\nu_2^{\tilde \mu}m^{\frac 1q‐\frac{1}{\tilde p}}n^{‐\frac 12} m^{\frac 12} k^{\frac{1}{\sigma}},
$$
$$
\min _{1\le j\le 5} \Phi_j(m, \, k, \, n) \underset{q,\sigma}{\asymp} \nu_1^{1‐\tilde \lambda}\nu_2^{\tilde \lambda}k^{\frac{1}{\sigma}‐\frac{1}{\tilde\theta}}n^{‐\frac 12} m^{\frac 1q} k^{\frac{1}{2}};
$$
now we apply Lemma \ref{lemma08}.
If $$m^{\frac{1}{p_1} ‐\frac{1}{p_2}} (n^{\frac 12} m^{‐\frac 12} k^{‐\frac{1}{\sigma}}) ^{\frac{1/\theta_1‐1/\theta_2}{1/2‐1/\sigma}} \le \frac{\nu_1}{\nu_2} \le (n^{\frac 12} m^{‐\frac 1q} k^{‐\frac{1}{2}}) ^{\frac{1/p_1‐1/p_2}{1/2‐1/q}} k^{\frac{1}{\theta_1}‐\frac{1}{\theta_2}},$$
then
$$
\min _{1\le j\le 5} \Phi_j(m, \, k, \, n) \underset{q,\sigma}{\asymp} \nu_1^{1‐\lambda}\nu_2^{\lambda}m^{\frac 1q‐\frac 1p}(n^{‐\frac 12} m^{\frac 12} k^{\frac{1}{\sigma}}) ^{\frac{1/\theta‐1/\sigma}{1/2‐1/\sigma}};
$$
we use Lemma \ref{lemma06}.

\vskip 0.2cm

\textbf{11. Case $p_1\in [2, \,  q]$, $\theta_1\in [1, \, 2]$, $p_2\in [1, \,  2]$, $\theta_2 \in [2, \, \sigma]$, $\tilde \lambda \le \tilde \mu$.} We prove that
$$
d_n(\nu_1B^{m,k}_{p_1,\theta_1} \cap \nu_2B^{m,k}_{p_2,\theta_2}, \, l^{m,k}_{q,\sigma}) \underset{q,\sigma}{\gtrsim} \min _{1\le j\le 4} \Phi_j(m, \, k, \, n).
$$

Since $\tilde \lambda \le \tilde \mu$, we have $\tilde p\le 2$, $\tilde \theta\le 2$.

{\bf Subcase $m^{2/q}k^{2/\sigma} \le n\le \min \{mk^{2/\sigma}, \, m^{2/q}k\}$.} Notice that
$$
(n^{\frac 12} m^{‐\frac 1q} k^{‐\frac{1}{\sigma}})^{\frac{1/p_1‐1/p_2}{1/2‐1/q}} \le 1\le (n^{\frac 12} m^{‐\frac 1q} k^{‐\frac{1}{\sigma}})^{\frac{1/\theta_1‐1/\theta_2}{1/2‐1/\sigma}}.
$$

If $\frac{\nu_1}{\nu_2} \le (n^{\frac 12} m^{‐\frac 1q} k^{‐\frac{1}{\sigma}})^{\frac{1/p_1‐1/p_2}{1/2‐1/q}}$ or $\frac{\nu_1}{\nu_2} \ge (n^{\frac 12} m^{‐\frac 1q} k^{‐\frac{1}{\sigma}})^{\frac{1/\theta_1‐1/\theta_2}{1/2‐1/\sigma}}$, we have, respectively,
$$
\min _{1\le j\le 4} \Phi_j(m, \, k, \, n) \underset{q,\sigma}{\asymp} \nu_1(n^{‐\frac 12} m^{\frac 1q} k^{\frac{1}{\sigma}}) ^{\frac{1/p_1‐1/q}{1/2‐1/q}},
$$
$$
\min _{1\le j\le 4} \Phi_j(m, \, k, \, n) \underset{q,\sigma}{\asymp} \nu_2(n^{‐\frac 12} m^{\frac 1q} k^{\frac{1}{\sigma}}) ^{\frac{1/\theta_2‐1/\sigma}{1/2‐1/\sigma}};
$$
we use Lemma \ref{lemma02}.
If $(n^{\frac 12} m^{‐\frac 1q} k^{‐\frac{1}{\sigma}})^{\frac{1/p_1‐1/p_2}{1/2‐1/q}}\le \frac{\nu_1}{\nu_2} \le 1$ or $1\le \frac{\nu_1}{\nu_2} \le (n^{\frac 12} m^{‐\frac 1q} k^{‐\frac{1}{\sigma}})^{\frac{1/\theta_1‐1/\theta_2}{1/2‐1/\sigma}}$, we get, respectively, 
$$
\min _{1\le j\le 4} \Phi_j(m, \, k, \, n) \underset{q,\sigma}{\asymp} \nu_1^{1‐\tilde \lambda}\nu_2^{\tilde \lambda} n^{‐\frac 12} m^{\frac 1q} k^{\frac{1}{\sigma}},
$$
$$
\min _{1\le j\le 4} \Phi_j(m, \, k, \, n) \underset{q,\sigma}{\asymp} \nu_1^{1‐\tilde \mu}\nu_2^{\tilde \mu} n^{‐\frac 12} m^{\frac 1q} k^{\frac{1}{\sigma}};
$$
we apply Lemma \ref{lemma09}.

{\bf Subcase $mk^{2/\sigma}\le n\le m^{2/q}k$.} Notice that
$$
m^{1/p_1‐1/p_2} \le 1\le (n^{\frac 12} m^{‐\frac 1q} k^{‐\frac{1}{\sigma}})^{\frac{1/\theta_1‐1/\theta_2}{1/2‐1/\sigma}}.
$$

If $\frac{\nu_1}{\nu_2} \ge 1$, as in the previous subcase, we use Lemmas \ref{lemma02} and \ref{lemma09}. Let $\frac{\nu_1}{\nu_2} \le 1$. If $\frac{\nu_1}{\nu_2} \le m^{\frac{1}{p_1}‐\frac{1}{p_2}}$, we have
$$
\min _{1\le j\le 4} \Phi_j(m, \, k, \, n) \underset{q,\sigma}{\asymp} \nu_1m^{\frac 1q‐\frac{1}{p_1}}n^{‐\frac 12} m^{\frac 12} k^{\frac{1}{\sigma}};
$$ 
we use Lemma \ref{lemma07}.
If $m^{\frac{1}{p_1}‐\frac{1}{p_2}} \le \frac{\nu_1}{\nu_2} \le 1$, then 
$$
\min _{1\le j\le 4} \Phi_j(m, \, k, \, n) \underset{q,\sigma}{\asymp} \nu_1^{1‐\tilde \lambda}\nu_2^{\tilde \lambda} n^{‐\frac 12} m^{\frac 1q} k^{\frac{1}{\sigma}};
$$
we use Lemma \ref{lemma10}.

{\bf Subcases $m^{2/q}k\le n\le mk^{2/\sigma}$ and $\max\{m^{2/q}k, \, mk^{2/\sigma}\}\le n \le \frac{mk}{2}$} are similar.

\end{document}